\documentclass[12pt]{amsart}
\usepackage{fullpage}
\usepackage{colonequals}
\usepackage{amsmath,amssymb,amsthm,url}
\usepackage{xcolor}
\usepackage[utf8]{inputenc}
\usepackage[english]{babel}
\usepackage[alphabetic]{amsrefs}
\usepackage{bbm}
\usepackage{enumerate}
\usepackage{tikz, tikz-cd}
\tikzset{
  red_vertex/.style={circle, thick, draw=red},
  blue_vertex/.style={circle, thick, draw=NavyBlue},
  purple_vertex/.style={circle, thick, draw=violet},
  vertex/.style={circle, thick, draw=black},
  edge/.style={thick},
  red_edge/.style={thick, draw=red},
  blue_edge/.style={thick, draw=NavyBlue},
  purple_edge/.style={thick, draw=violet}
}
\usetikzlibrary{shapes.geometric, positioning, calc, backgrounds, decorations.pathreplacing}

\usepackage{graphicx}
\usepackage{centernot}
\usepackage{mathrsfs}
\usepackage[colorlinks=true, pdfstartview=FitH, linkcolor=blue, citecolor=blue, urlcolor=blue]{hyperref}

\allowdisplaybreaks

\newtheorem{theoremA}{Theorem}

\newtheorem{theorem}{Theorem}[section]
\newtheorem{lemma}[theorem]{Lemma}
\newtheorem{prop}[theorem]{Proposition}
\newtheorem{corollary}[theorem]{Corollary}

\newcommand{\set}[1]{\{#1\}}

\theoremstyle{definition}
\newtheorem{defn}[theorem]{Definition}

\newtheorem{remark}[theorem]{Remark}
\newtheorem{ex}[theorem]{Example}

\newcommand{\abs}[1]{\lvert#1\rvert}
\newcommand{\B}{\mathcal{B}}
\newcommand{\cP}{\mathcal{P}}

\newcommand{\Gbar}{\overline{G}}
\newcommand{\Htil}{\widetilde{H}}

\newcommand{\dbledge}{
  \raisebox{0.6ex}{
    \resizebox{1.4ex}{!}{
      \begin{tikzpicture}[baseline=0ex]
          \draw[line width=3pt] (0,0) to[bend left=75] (1.4,0);
          \draw[line width=3pt] (0,0) to[bend right=75] (1.4,0);
          \fill (0,0) circle (0.23); 
          \fill (1.4,0) circle (0.23);
      \end{tikzpicture}%
    }
  }
}

\newcommand{\BellMulti}[2]{\mathcal{B}^{\mkern-6mu\protect\dbledge}_{#1}\mkern-6mu(#2)}

\newcommand{\BellMultiIntro}[2]{\mathcal{B}^{\mkern-6mu\protect\dbledge}_{#1}\mkern-2mu(#2)}

\newcommand{\zap}[1]{\langle #1 \rangle}

\title{Bell coloring graphs: realizability and reconstruction}
\author{Shamil Asgarli}
\address{Department of Mathematics \& Computer Science \\ Santa Clara University \\ CA 95053 \\ USA}
\email{sasgarli@scu.edu}

\author{Sara Krehbiel}
\address{Department of Mathematics \& Computer Science \\ Santa Clara University \\ CA 95053 \\ USA}
\email{skrehbiel@scu.edu}

\author{Simon MacLean}
\address{Department of Mathematics \& Computer Science \\ Santa Clara University \\ CA 95053 \\ USA}
\email{smaclean@scu.edu}


\begin{document}

\begin{abstract}
Given a graph $G$, the Bell $k$-coloring graph $\mathcal{B}_k(G)$ has vertices given by partitions of $V(G)$ into $k$ independent sets (allowing empty parts), with two partitions adjacent if they differ only in the placement of a single vertex. We first give a structural classification of cliques in Bell coloring graphs. We then show that all trees and all cycles arise as Bell coloring graphs, while $K_4-e$ is not a Bell coloring graph and, more generally, $K_n-e$ is not an induced subgraph of any Bell coloring graph whenever $n \ge 6$. We also prove two reconstruction results: the Bell $3$-coloring graph is a complete invariant for trees, and the Bell $n$-coloring multigraph determines any graph up to universal vertices.
\end{abstract}

\maketitle

\section{Introduction}\label{sec:intro}

Graph coloring is a vibrant area of research that bridges theoretical questions and practical applications. Several excellent references survey both the classic landscape and modern frontiers; see, for example, the monograph \cite{JT95}, the themed collection \cite{BW15}, and the recent text \cite{Cra24}, which focuses on contemporary techniques. One research theme that has gained momentum is the study of not only the colorings themselves but also the \emph{reconfiguration graphs} that connect them. In a standard $k$-coloring graph, vertices represent proper $k$-colorings, and edges correspond to changing the color of a single vertex. These graphs retain much of the structure of the underlying graph. Indeed, recent work has shown that coloring graphs can serve as complete graph invariants \cites{HSTT24, BBHvdHHP25, AKL25}.

In this paper, we study a compressed version of a coloring graph obtained by treating the color classes as indistinguishable. In this model, we retain the partition of vertices into independent sets but `forget' the specific color labels assigned to each set. Given a base graph $G$ and $k\in\mathbb{N}$, we define the \emph{Bell $k$-coloring graph} $\mathcal{B}_k(G)$ as follows.
\begin{itemize}
    \item The vertices of $\mathcal{B}_k(G)$ are partitions of the vertex set $V(G)$ into $k$ independent sets (some possibly empty); we call such partitions \emph{stable} $k$-partitions.
    \item Two distinct partitions $P_1$ and $P_2$ are adjacent if they differ only by the placement of a single vertex $v\in V(G)$, formalized as $P_1-v = P_2-v$.
\end{itemize}
The notation $P-v$ means the partition $P$ with the vertex $v$ removed from its part; that is, $P-v$ is the restriction of $P$ to $G-v$. Figures~\ref{fig:Bell3-K13}~and~\ref{fig:Bell3-K3-plus-K1} depict examples of Bell coloring graphs with vertices labeled to indicate the corresponding partitions and edges labeled to indicate the vertex or vertices responsible for each edge.

The name `Bell coloring graph' is motivated by the Bell numbers, which count the total number of set partitions. This object, also studied by Haas \cite{Haa12} as the `isomorphic color graph', interpolates between enumerative invariants (graphical Bell numbers \cite{DP09}) and reconfiguration structure. In contrast to standard coloring graphs, only a handful of papers (e.g., \cite{Haa12}, \cite{FM14}) address Bell coloring graphs.

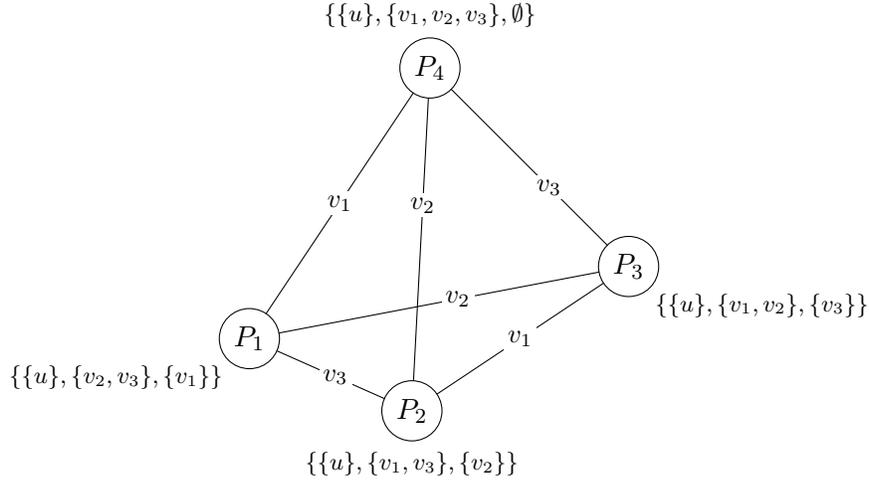
\begin{figure}[h]
\centering
\begin{tikzpicture}[
    vertex/.style={circle, draw, fill=white, minimum size=8mm, inner sep=0pt, font=\small},
    edge label/.style={midway, fill=white, font=\footnotesize, inner sep=1.5pt},
    part label/.style={font=\scriptsize, color=black, align=center},
    x={(-0.6cm, -0.4cm)}, y={(1cm, 0cm)}, z={(0cm, 1cm)},
    scale=1.2
]

\coordinate (P1) at (0, -2, 0);
\coordinate (P2) at (2, 1, 0);
\coordinate (P3) at (-2, 1, 0);
\coordinate (P4) at (0, 0, 3);

\draw (P1) -- (P2) node[edge label, pos=0.53] {$v_3$};
\draw (P1) -- (P3) node[edge label, pos=0.55] {$v_2$}; 
\draw (P2) -- (P3) node[edge label] {$v_1$}; 
\draw (P4) -- (P1) node[edge label, pos=0.5] {$v_1$};
\draw (P4) -- (P2) node[edge label, pos=0.4] {$v_2$};
\draw (P4) -- (P3) node[edge label, pos=0.6] {$v_3$};

\node[vertex] at (P1) {$P_1$};
\node[part label, below left=0.3cm of P1] {$\{\{u\}, \{v_2, v_3\}, \{v_1\} \}$};

\node[vertex] at (P2) {$P_2$};
\node[part label, below=0.4cm of P2] {$\{\{u\}, \{v_1, v_3\}, \{v_2\} \}$};

\node[vertex] at (P3) {$P_3$};
\node[part label, below right=0.3cm of P3] {$\{\{u\}, \{v_1, v_2\}, \{v_3\} \}$};

\node[vertex] at (P4) {$P_4$};
\node[part label, above=0.4cm of P4] {$\{\{u\}, \{v_1, v_2, v_3\}, \emptyset\}$};

\end{tikzpicture}
\caption{Illustration of $\mathcal{B}_3(K_{1,3})\cong K_4$. The claw graph $K_{1, 3}$ has vertex set $\{u, v_1, v_2, v_3\}$ and three edges $uv_i$ for $i=1,2,3$. Edge labels indicate the vertices responsible for the adjacency.} 
\label{fig:Bell3-K13}
\end{figure}

\begin{figure}[h]
\centering
\begin{tikzpicture}[
    vertex/.style={circle, draw, fill=white, minimum size=8mm, inner sep=0pt, font=\small},
    edge label/.style={midway, fill=white, font=\footnotesize, inner sep=1.5pt},
    part label/.style={font=\scriptsize, color=black, align=center},
    x={(-0.6cm, -0.4cm)}, y={(1cm, 0cm)}, z={(0cm, 1cm)},
    scale=1.2
]

\coordinate (P1) at (0, -2, 0);
\coordinate (P2) at (2, 1, 0);
\coordinate (P3) at (-2, 1, 0);
\coordinate (P4) at (0, 0, 3);

\draw (P1) -- (P2) node[edge label, pos=0.53] {$w$};
\draw (P1) -- (P3) node[edge label, pos=0.55] {$w$}; 
\draw (P2) -- (P3) node[edge label] {$w$}; 
\draw (P4) -- (P1) node[edge label, pos=0.5] {$w, v_1$};
\draw (P4) -- (P2) node[edge label, pos=0.4] {$w, v_2$};
\draw (P4) -- (P3) node[edge label, pos=0.6] {$w, v_3$};

\node[vertex] at (P1) {$P_1$};
\node[part label, below left=0.3cm of P1] {$\{\{w, v_1\}, \{v_2\}, \{v_3\}, \emptyset\}$};

\node[vertex] at (P2) {$P_2$};
\node[part label, below=0.4cm of P2] {$\{\{w, v_2\}, \{v_1\}, \{v_3\}, \emptyset\}$};

\node[vertex] at (P3) {$P_3$};
\node[part label, below right=0.3cm of P3] {$\{\{w, v_3\}, \{v_1\}, \{v_2\}, \emptyset \}$};

\node[vertex] at (P4) {$P_4$};
\node[part label, above=0.4cm of P4] {$\{\{w\}, \{v_1\}, \{v_2\}, \{v_3\}\}$};

\end{tikzpicture}
\caption{Illustration of $\mathcal{B}_4(K_{3}\sqcup K_1)\cong K_4$. The graph $K_3\sqcup K_1$ has vertex set $\{v_1, v_2, v_3, w\}$ with edges $v_1v_2, v_2v_3, v_3v_1$. Edge labels indicate the vertices responsible for the adjacency.} 
\label{fig:Bell3-K3-plus-K1}
\end{figure}
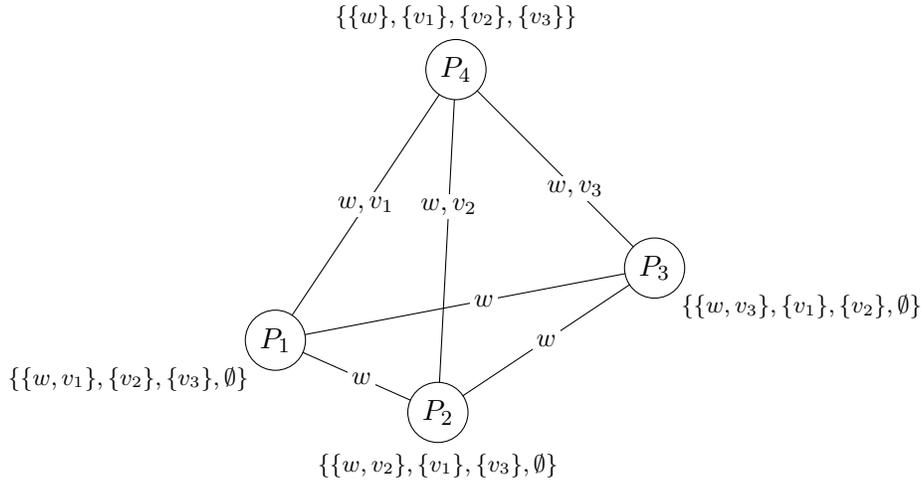

Our work is motivated in part by the question of which graphs can arise as Bell coloring graphs. Every coloring graph is a Bell coloring graph because $\B_k(G\sqcup K_k) \cong \mathcal{C}_k(G)$ for any graph $G$ and integer $k\ge 1$ \cite{FM14}*{Proposition 2.5}. However, the class of graphs realizable as standard coloring graphs is quite limited. For example, $K_1$ and $K_2$ are the \emph{only} trees that are realizable as coloring graphs, and $C_3$, $C_4$, and $C_6$ are the \emph{only} cycles that are realizable as coloring graphs \cite{BFHRS16}*{Theorems 7 and 23}. In contrast, we prove the following: 

\begin{theoremA}[Theorem~\ref{thm:trees-cycles}]
All trees and all cycles are realizable as Bell coloring graphs.
\end{theoremA}

We prove Theorem~\ref{thm:trees-cycles} by showing that a more general class of graphs, namely certain reconfiguration graphs of matchings, are Bell coloring graphs. 

The papers \cites{BFHRS16, ABFR17} examine forbidden subgraphs in standard coloring graphs. To treat the analogous question for Bell coloring graphs, we first obtain a structural description of their cliques (see Theorem~\ref{thm:clique_classification}), showing that every clique belongs to one of two explicit families. This classification allows us to construct an infinite family of forbidden induced subgraphs of Bell coloring graphs.

\begin{theoremA}[Theorem~\ref{thm:forbidden}]
The graph $K_{6}-e$ is not an induced subgraph of any Bell coloring graph. Hence, the set of graphs $K_n-e$ for $n\ge 6$ is an infinite family of forbidden induced subgraphs.
\end{theoremA}

Another natural question is whether the structure of the reconfiguration graph uniquely determines the base graph. This is known for standard coloring graphs; in particular, if $T$ is a tree, $\mathcal{C}_3(T)$ uniquely determines $T$ (\cite{AKL25}*{Theorem 1.1} or \cite{BBHvdHHP25}*{Theorem 1.2}). Having lost the color labels, $\mathcal{B}_3(T)$ has less symmetry than the $3$-coloring graph $\mathcal{C}_3(T)$. Despite this compression, our next result guarantees that a tree $T$ can be reconstructed uniquely from its Bell $3$-coloring graph. Let $\mathsf{Trees}$ and $\mathsf{Graphs}$ denote the sets of isomorphism classes of all (finite) trees and graphs, respectively. 

\begin{theoremA}[Theorem~\ref{thm:tree-reconstruction}]
The map $\mathsf{Trees} \to \mathsf{Graphs}$ given by $T\mapsto \mathcal{B}_3(T)$ is injective.
\end{theoremA}

We also find it useful to work with a multigraph variant, the \emph{Bell coloring multigraph}, denoted $\BellMulti{k}{G}$, which uses multiple edges to encode how many vertices are responsible for each adjacency. If $P$ and $P'$ are two stable partitions of $G$, then in the Bell coloring graph, $P$ and $P'$ are adjacent whenever there exists a vertex $v\in V(G)$ with $P-v = P'-v$. Thus, multiple vertices may witness the same adjacency, but $\B_k(G)$ still records only a single edge between $P$ and $P'$. In the multigraph version, we instead add a separate parallel edge between $P$ and $P'$ for each distinct vertex $v$ such that $P-v = P'-v$. 

The additional information means that Bell coloring multigraphs serve as a more refined graph invariant. For example, Figures~\ref{fig:Bell3-K13}~and~\ref{fig:Bell3-K3-plus-K1} show that $\B_3(K_{1,3})\cong \B_4(K_3\sqcup K_1)$. However, each edge incident to the partition $P_4$ in the latter graph is witnessed by two vertices, so 
$\BellMulti{3}{K_{1,3}}\not\cong\BellMulti{4}{K_3\sqcup K_{1}}$. 

Yet even a Bell coloring multigraph cannot serve as a complete graph invariant. For instance, $\BellMulti{3}{K_{1,3}}\cong \BellMulti{2}{\overline{K}_3}$. This ambiguity arises because the center vertex in $K_{1,3}$ is adjacent to every other vertex. When a vertex $v$ is adjacent to every vertex in $G-v$, we call $v$ a \emph{universal vertex}. 
If $G$ has a universal vertex $w$, then $\B_{k+1}(G) \cong \B_k(G-w)$ for every $k\in\mathbb{N}$ due to the natural bijection between stable $k$-partitions of $G-w$ and stable $(k+1)$-partitions of $G$ in which $\{w\}$ forms an additional singleton part.

To handle this, let $U(G)$ denote the set of universal vertices of $G$. We define the \emph{core} of $G$, denoted $G^{\circ}$, as the induced subgraph $G[V(G)\setminus U(G)]$. We define the equivalence relation $\sim_{\text{uni}}$ on $\mathsf{Graphs}$ as follows: $G\sim_{\text{uni}} G'$ if their cores are isomorphic, $G^{\circ} \cong (G')^{\circ}$.
The equivalence classes of graphs under this relation form the set $\mathsf{Graphs}^{\circ} = \mathsf{Graphs}/\!\sim_{\text{uni}}$.

Our final result shows that a base graph of order $n$ can be reconstructed from its Bell $n$-coloring multigraph up to universal vertices:

\begin{theoremA}[Theorem~\ref{thm:general-reconstruction-multigraph:intro}]
The map $\mathsf{Graphs}^{\circ} \to \mathsf{Multigraphs}$ given by $G\mapsto \BellMultiIntro{|V(G)|}{G}$  is injective.
\end{theoremA}

\textbf{Outline of the paper.} Section~\ref{sec:preliminaries} introduces Bell $k$-coloring graphs, gives small examples, and develops basic tools for understanding adjacency, including a description of edges that are realized by two vertices. In Section~\ref{sec:cliques}, we classify cliques in Bell coloring graphs and apply this to show that $K_4-e$ is not a Bell coloring graph and that $K_n-e$ is not an induced subgraph of any Bell coloring graph for $n\ge 6$. Section~\ref{sec:matchings} introduces a matching reconfiguration graph and shows that, for triangle-free graphs, it coincides with a Bell coloring graph; this yields realizations of all trees and all cycles as Bell coloring graphs. Section~\ref{sec:tree-reconstruction} proves that trees are reconstructible from their Bell
$3$-coloring graphs. Section~\ref{sec:reconstruction-multigraph} proves the multigraph reconstruction theorem: the Bell $n$-coloring multigraph determines $G$ up to universal vertices.

\section{Preliminaries}\label{sec:preliminaries}

\subsection{Notation and conventions} Let $\mathbb{N}=\{1, 2, 3, \dots\}$ denote the set of positive integers. For sets $X$ and $Y$, we use the notation $X\setminus Y=\{x\in X \ | \ x\notin Y\}$ for the set difference and $X\cup Y$ for set union; in the special case when $Y=\{v\}$ is a singleton, we also denote $X\setminus\{v\}$ by $X-v$ and $X\cup\{v\}$ by $X+v$. Given a graph $G$, we let $\overline{G}$ denote its complement. We write $G\sqcup H$ for the disjoint union of graphs $G$ and $H$. 

\begin{defn}\label{def:stable-partition}
A \emph{stable $k$-partition} of a graph $G$ is a multiset $P=\{V_1,\dots,V_k\}$ of independent sets that partition $V(G)$, where we allow some of the parts $V_i$ to be empty. 
\end{defn}

Let $\cP_k(G)$ denote the set of all stable $k$-partitions of $G$. For $v\in V(G)$, we write
\[
P-v \colonequals \{V_i-v : 1\leq i\leq k\},
\]
viewed as a multiset of $k$ subsets.

A Bell $k$-coloring graph $\B_k(G)$ is the graph whose vertex set is $\cP_k(G)$. Two vertices $P_1, P_2 \in \cP_k(G)$ are adjacent if and only if $P_1-v=P_2-v$ for some $v\in V(G)$. 

Throughout, partitions $P, Q, R \in \cP_k(G)$ are treated as \emph{unordered multisets} of parts. We write $P \sim Q$ to denote adjacency in $\B_k(G)$. To specify the vertex responsible for the adjacency, we use the following notation:
\[
P \sim_v Q \iff P \sim Q \text{ and } P-v = Q-v.
\]
We say that a vertex $v\in V(G)$ is \emph{responsible} for an edge $PQ$ in $\B_k(G)$ if $P\sim_v Q$. Equivalently, the edge $PQ$ is \emph{realized} by $v$. The condition $P-v=Q-v$ implies that $P$ and $Q$ agree on the structure of the partition when restricted to $V(G)\setminus\{v\}$. 

We record the following useful fact, which we use several times in the paper: if $P_1\sim_v P_2$ and $P_2\sim_v P_3$, then $P_1$ and $P_3$ agree on $V(G)\setminus\{v\}$, so $P_1\sim_v P_3$.

\subsection{Example Bell coloring graphs} We begin with a few illustrative examples.

\begin{ex} Consider $\B_2(\overline{K}_2)$ where $V(\overline{K}_2)=\{1,2\}$. The stable 2-partitions are $R_1 = \{\{1,2\}, \emptyset\}$ and $R_2 = \{\{1\}, \{2\}\}$.
We have $R_1-1 = \{\{2\}, \emptyset\}$ and $R_2-1 = \{\emptyset, \{2\}\}$, so $R_1 \sim_1 R_2$. Similarly $R_1 \sim_2 R_2$. Thus $\B_2(\overline{K}_2) \cong K_2$.
\end{ex}

\begin{ex}\label{ex:K3}
Consider $G=K_3 \sqcup K_1$ on vertices $\{1,2,3,4\}$, where $\{1,2,3\}$ form the $K_3$ and the vertex $4$ is isolated. We examine $\B_3(G)$. Since $\{1,2,3\}$ is a clique, any stable 3-partition must place them in different parts. The vertex $4$ can join any part.
\begin{align*}
P_1 &= \{\{1,4\}, \{2\}, \{3\}\}, \\
P_2 &= \{\{1\}, \{2,4\}, \{3\}\}, \\
P_3 &= \{\{1\}, \{2\}, \{3,4\}\}.
\end{align*}
These are the only stable 3-partitions of $G$. We check adjacencies: $P_1-4 = P_2-4 = P_3-4 = \{\{1\}, \{2\}, \{3\}\}$. Thus, we have $P_1 \sim_4 P_2$, $P_2 \sim_4 P_3$, and $P_3 \sim_4 P_1$, so $\B_3(G) \cong K_3$.
\end{ex}

\begin{ex}\label{ex:P4}
Consider the path graph $G$ on vertices $\{1,2,3,4\}$ with edges 12, 23, 34. The stable $3$-partitions are:
\[
\begin{aligned}
R_1 &= \{\{1,3\}, \{2\}, \{4\}\}, & R_2 &= \{\{1,4\}, \{2\}, \{3\}\},\\
R_3 &= \{\{1\}, \{2,4\}, \{3\}\}, & R_4 &= \{\{1,3\}, \{2,4\}, \emptyset\}.
\end{aligned}
\]
A direct check shows that $R_1 \sim_1 R_2$, $R_2 \sim_4 R_3$, $R_3 \sim_1 R_4$, and $R_4 \sim_4 R_1$. The graph $\B_3(P_4)$ is the cycle $C_4$; see Figure~\ref{fig:B3_P4}.
\end{ex}

\begin{figure}[h]
    \centering
    \begin{tikzpicture}[
        partition/.style={rectangle, draw, rounded corners, thick, inner sep=5pt, align=center, fill=white},
        edge label/.style={midway, fill=white, font=\small, inner sep=1.5pt},
        scale=1.1
    ]

    \node[partition] (R1) at (0, 2) {$\{1,3\}, \{2\}, \{4\}$};
    
    \node[partition] (R2) at (5, 2) {$\{1,4\}, \{2\}, \{3\}$};
    
    \node[partition] (R3) at (5, 0) {$\{1\}, \{2,4\}, \{3\}$};

    \node[partition] (R4) at (0, 0) {$\{1,3\}, \{2,4\}, \emptyset$};

    \draw[thick] (R1) -- (R2) node[edge label] {$1$};

    \draw[thick] (R2) -- (R3) node[edge label] {$4$};

    \draw[thick] (R3) -- (R4) node[edge label] {$1, 3$};

    \draw[thick] (R4) -- (R1) node[edge label] {$2, 4$};

    \end{tikzpicture}
    \caption{The Bell $3$-coloring graph of the path $P_4$ is the cycle graph $C_4$. Each edge incident to $R_4=\{\{1,3\}, \{2, 4\}, \emptyset\}$ is doubly realized: it can be obtained by moving either of two nonadjacent vertices (using the empty part in $R_4$), so the same pair of partitions is adjacent via two distinct vertex moves.}
    \label{fig:B3_P4}
\end{figure}
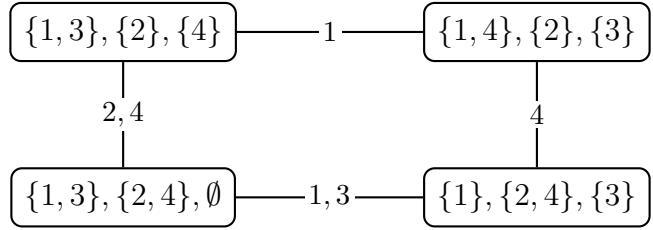

\subsection{Multiple edges}

As we see in Figure~\ref{fig:B3_P4}, an interesting structural feature of Bell coloring graphs is that two partitions might be adjacent via more than one vertex simultaneously. The following lemma characterizes this phenomenon.

\begin{lemma}[Characterization of doubly realized edges]\label{lem:double_edge}
Let $P, Q \in \cP_k(G)$ be distinct partitions. If $P \sim_v Q$ and $P \sim_w Q$ for distinct vertices $v, w \in V(G)$, then $v$ and $w$ are nonadjacent in $G$. Furthermore, $P$ and $Q$ decompose as:
\begin{align*}
    P=\{\{v,w\}, \emptyset\} \cup R, \quad \text{ and } \quad Q=\{\{v\}, \{w\}\} \cup R
\end{align*}
for some multiset $R$ consisting of $k-2$ independent sets.
\end{lemma}

\begin{proof}
Since $P\sim_v Q$, the partitions coincide when $v$ is removed. Thus, there exist sets $V_1, V_2 \subseteq V(G)\setminus\{v\}$ and a collection of parts $R'$ such that:
\begin{align*}
    P &= \{V_1 \cup \{v\}, V_2\} \cup R', \\
    Q &= \{V_1, V_2 \cup \{v\}\} \cup R'.
\end{align*}
Since $P\neq Q$, we must have $V_1 \neq V_2$. Now consider the condition $P \sim_w Q$. If $w$ were in some part $C \in R'$, then $P-w=Q-w$ would imply $\{V_1 \cup \{v\}, V_2\} = \{V_1, V_2 \cup \{v\}\}$, forcing $V_1=V_2$, a contradiction. Thus $w \in V_1$ or $w \in V_2$. Without loss of generality, assume $w \in V_1$.

Write $V_1 = V_1' \cup \{w\}$. Then:
\begin{align*}
    P-w &= \{V_1' \cup \{v\}, V_2\} \cup R', \\
    Q-w &= \{V_1', V_2 \cup \{v\}\} \cup R'.
\end{align*}
Equality forces the multisets to be the same: $\{V_1' \cup \{v\}, V_2\} = \{V_1', V_2 \cup \{v\}\}$. This implies $V_1'=V_2$. For the parts of $P$ to be disjoint, we must have $(V_1' \cup \{w, v\}) \cap V_2 = \emptyset$. Since $V_1'=V_2$, this requires $V_1'=\emptyset$, so $V_2=\emptyset$. Thus $P=\{\{v,w\}, \emptyset\} \cup R'$ and $Q=\{\{v\}, \{w\}\} \cup R'$. Finally, since $\{v,w\}$ is a part in $P$, $v$ and $w$ must be nonadjacent in $G$.
\end{proof}

We encountered this double–edge phenomenon in the smallest case
$\B_2(\overline{K}_2)\cong K_2$: there are two partitions
$P=\{\{u,v\},\emptyset\}$ and $Q=\{\{u\},\{v\}\}$, and both $u$ and $v$
are responsible for the adjacency $P\sim Q$. Thus, the unique edge of
$\B_2(\overline{K}_2)$ acts like a doubly realized edge, corresponding exactly to the
special case of Lemma~\ref{lem:double_edge} with $R=\emptyset$.

\begin{remark}\label{rem:double-edge-responsible}
Let $P, Q\in\cP_k(G)$ be distinct partitions. Then $PQ$ is a doubly realized edge with distinct $v,w\in V(G)$ such that $P\sim_v Q$ and $P\sim_w Q$ if and only if there is a collection $R$ of $k-2$ parts such that, up to relabeling of parts,
\[
P=\{\{v,w\},\emptyset\}\cup R
\quad\text{and}\quad
Q=\{\{v\},\{w\}\}\cup R.
\]
In particular, at most two vertices of $G$ can be responsible for a single edge $PQ$ in $\B_k(G)$. Indeed, if three distinct vertices $v,w,x$ were all responsible for $PQ$, then applying Lemma~\ref{lem:double_edge} to the pairs $(v,w)$ and $(v,x)$ would force two different decompositions of $P$ of the above form, contradicting the uniqueness of the part containing $v$. Equivalently, for every edge $PQ$ in $\B_k(G)$, the set
$\{u\in V(G): P\sim_u Q\}$ has size at most $2$.
\end{remark}

\section{Cliques in Bell coloring graphs}\label{sec:cliques}

The relationship between colorings that give rise to cliques in a $k$-coloring graph is straightforward: a copy of $K_m$ appears in $\mathcal{C}_k(G)$ if and only if it arises by recoloring a single vertex in $m$ distinct ways \cite{BFHRS16}*{Lemma 8}. In Subsection~\ref{sec:clique-classification}, we provide several definitions and lemmas leading up to Theorem~\ref{thm:clique_classification}, which enumerates the types of cliques present in Bell coloring graphs. In contrast to standard coloring graphs, small cliques in $\B_k(G)$ need not be realized by a single anchor vertex. However, for any clique $K_m$ with $m\ge 5$ in a Bell coloring graph, there is a vertex $u\in V(G)$ such that $P_i\sim_u P_j$ for every edge $P_iP_j$ of the clique. In Subsection~\ref{sec:forbidden}, we use our classification result to prove that certain nearly complete graphs cannot appear as Bell coloring graphs or as induced subgraphs of Bell coloring graphs.

\subsection{Definitions and clique classification theorem}\label{sec:clique-classification}

To classify the cliques of $\B_k(G)$, we group them according to which vertices of $G$ realize their edges. First, we differentiate between cliques in which each edge is realized by the same vertex and all other cliques. We call the former an $S$-clique since it is realized by a single vertex, and the latter a $T$-clique, since the smallest such example is realized by three distinct vertices.

\begin{defn}\label{def:S-vs-T}
An {\em $S$-clique} in the Bell coloring graph $\B_k(G)$ is a clique $\mathcal{K}=\{P_1, \dots, P_m\}$ that admits an \emph{anchor} vertex $u$ such that $P_i - u = P_j - u$ for all $1\leq i<j\leq m$; equivalently, for every edge $P_iP_j$ of the clique we have $P_i\sim_u P_j$. 

A {\em $T$-clique} in a Bell coloring graph is any clique that is not an $S$-clique.
\end{defn}

In standard coloring graphs, an anchor vertex for a clique is solely responsible for each edge. By contrast, Figure~\ref{fig:hybrid_triangle} depicts an $S$-triangle in which one of the edges has two vertices responsible for it.

\begin{figure}[h]
    \centering
    \begin{tikzpicture}[
        vertex/.style={circle, draw, minimum size=7mm, inner sep=0pt, font=\small},
        edge label/.style={midway, fill=white, font=\footnotesize, inner sep=1.5pt},
        part label/.style={font=\scriptsize, color=black} 
    ]

    \begin{scope}[local bounding box=hybrid]
        \node[vertex] (P1) at (90:2.0) {$P_1$};
        \node[vertex] (P2) at (210:2.0) {$P_2$};
        \node[vertex] (P3) at (330:2.0) {$P_3$};
        \node[part label, above=0.1cm of P1] {$\{\{v,w\}, \{x,y\}\}$};
        \node[part label, below left=0.1cm of P2] {$\{\{v\}, \{w\}, \{x,y\}\}$};
        \node[part label, below right=0.1cm of P3] {$\{\{w\}, \{v,x,y\}\}$};
        \draw (P1) -- (P2) node[edge label] {$v, w$};
        \draw (P2) -- (P3) node[edge label] {$v$};
        \draw (P3) -- (P1) node[edge label] {$v$};
        \node at (0,-2) {Hybrid triangle};
    \end{scope}

    \end{tikzpicture}
    \caption{An example of an $S$-clique in $\B_3(G)$ (where $G$ is $K_{1, 2}$ with center $w$ plus an isolated vertex $v$) that exhibits hybrid behavior. The vertex $v$ is responsible for all three edges (making it an $S$-triangle), but both $v$ and $w$ are responsible for the edge $P_1P_2$, making it a doubly realized edge.}
    \label{fig:hybrid_triangle}
\end{figure}
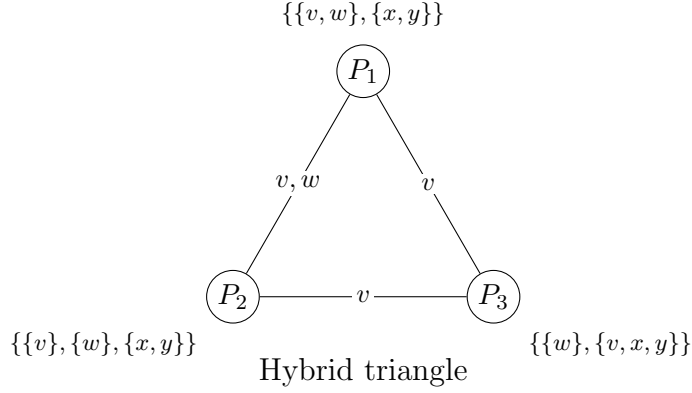

We present some lemmas that distinguish $S$-cliques from $T$-cliques before enumerating special types of $T$-cliques. The following lemma establishes that although an $S$-clique may have doubly realized edges, there is a unique vertex that realizes all edges whenever the clique has at least three vertices.

\begin{lemma}[Anchor uniqueness]\label{lem:anchor_uniqueness}
If $\mathcal{K}$ is a clique of order at least $3$ in the Bell coloring graph of $G$
and there exists a vertex $a\in V(G)$ such that every edge of $\mathcal{K}$ is realized by $a$, then $a$ is unique. That is, every $S$-clique of order at least $3$ has a well-defined anchor vertex.
\end{lemma}

\begin{proof}
Suppose, to the contrary, that an $S$-clique $\mathcal{K}$ with $|\mathcal{K}|\geq 3$ has two distinct anchors $a$ and $b$. Let $P_1, P_2, P_3 \in \mathcal{K}$ be distinct partitions.
Since the edge $P_1P_2$ is realized by both $a$ and $b$, Lemma~\ref{lem:double_edge} implies that:
\[
P_1 = \{\{a, b\}, \emptyset\} \cup R
\qquad\text{and}\qquad
P_2 = \{\{a\}, \{b\}\} \cup R
\]
for some common restriction $R$ on $V(G)\setminus\{a,b\}$. Applying the same argument to $P_1$ and $P_3$ forces $P_3 = \{\{a\}, \{b\}\} \cup R$, so $P_3 = P_2$, a contradiction.
\end{proof}

We show that every $T$-triangle has three distinct vertices responsible for its edges.

\begin{lemma}\label{lem:triangle_realizers}
Let $\Delta=\{P_1,P_2,P_3\}$ be a triangle in $\B_k(G)$, and for each $1\le i<j\le 3$, let
\[
V_{ij} \colonequals \{v\in V(G) : P_i \sim_v P_j\}.
\]
Then exactly one of the following holds:
\begin{enumerate}
    \item\label{item:S-triangle} $V_{12}\cap V_{23}\cap V_{31}\neq\emptyset$, in which case $\Delta$ is an $S$-triangle.
    \item\label{item:T-triangle} The sets $V_{12}, V_{23}, V_{31}$ are pairwise disjoint singletons, in which case $\Delta$ is a $T$-triangle.
\end{enumerate}
\end{lemma}

\begin{proof}
If $P_1\sim_v P_2$ and $P_2\sim_v P_3$, then $P_1\sim_v P_3$. Hence if $u\in V_{12}\cap V_{23}$, transitivity implies $u\in V_{31}$, so $u$ realizes every edge of $\Delta$ and is a common anchor. In this case, $\Delta$ is an $S$-triangle, and we are in alternative~\eqref{item:S-triangle}.

Now assume that $\Delta$ is not an $S$-triangle. By definition, $\Delta$ is then a $T$-triangle, and alternative~\eqref{item:S-triangle} fails. No vertex can lie in two of the sets $V_{ij}$. Indeed, if $u\in V_{12}\cap V_{23}$, the previous paragraph would show that $u$ is a common anchor. Thus $V_{12},V_{23},V_{31}$ are pairwise disjoint. Since each $P_iP_j$ is an edge, each $V_{ij}$ is nonempty, and by Lemma~\ref{lem:double_edge} we have $|V_{ij}|\in\{1,2\}$.

Suppose for contradiction that $\Delta$ is a $T$-triangle but some edge is doubly realized, say $|V_{12}|=2$ with $V_{12}=\{a,b\}$. By Lemma~\ref{lem:double_edge}, we may assume
\[
P_1 = \{\{a,b\}, \emptyset\} \cup R, \qquad
P_2 = \{\{a\}, \{b\}\} \cup R
\]
for some multiset $R$. Let $v\in V_{13}$ and $w\in V_{23}$. Since the sets $V_{12},V_{23},V_{31}$ are pairwise disjoint, we have $v,w\notin\{a,b\}$.
Because $v \notin \{a,b\}$, the adjacency $P_1 \sim_v P_3$ forces $P_3$ to retain the part $\{a,b\}$. Similarly, because $w \notin \{a,b\}$, the adjacency $P_2 \sim_w P_3$ forces $P_3$ to retain the singletons $\{a\}$ and $\{b\}$. This is impossible: $P_3$ cannot simultaneously contain $\{a,b\}$ and $\{a\},\{b\}$.

Thus, in a triangle with no common anchor (that is, a $T$-triangle), no edge can be doubly realized. Since each $V_{ij}$ is nonempty and the three sets are pairwise disjoint, it follows that $V_{12},V_{23},V_{31}$ are pairwise disjoint singletons. This is exactly alternative~\eqref{item:T-triangle}.
\end{proof}

\begin{corollary}\label{cor:triangle_double_edge}
If $\{P_1,P_2,P_3\}$ is a triangle in $\B_k(G)$ with a doubly realized edge, then the triangle is an $S$-clique. Equivalently, every $T$-triangle has all three edges singly realized by distinct vertices.
\end{corollary}

Next, we define special types of small $T$-cliques. Figures~\ref{fig:triangle_types}~and~\ref{fig:k4_structures} depict these types of cliques in general. 

\begin{defn}\label{def:T-types}
    A {\em cyclic $T$-triangle} is a 3-clique $\{P_1,P_2,P_3\}$ with edges realized by $v_1,v_2,v_3$ such that $P_i$ contains $\{v_i\}$ as a singleton part for $i=1,2,3$.

    A {\em radial $T$-triangle} is a 3-clique $\{P_1,P_2,P_3\}$ with edges realized by $v_1,v_2,v_3$ such that $P_1$ contains $\{v_1,v_2,v_3\}$ as a part.
    
    A {\em fused $T$-tetrahedron} is a 4-clique consisting of four $T$-triangles.

    A {\em split $T$-tetrahedron} is a 4-clique consisting of three $S$-triangles and one $T$-triangle. 
\end{defn}

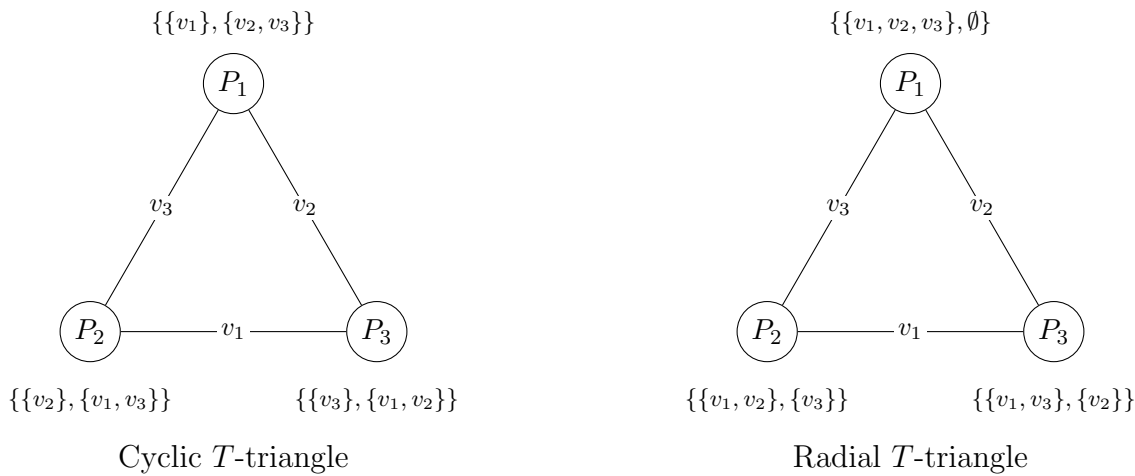
\begin{figure}[h]
    \centering
    \begin{tikzpicture}[
        vertex/.style={circle, draw, minimum size=8mm, inner sep=0pt, font=\small},
        edge label/.style={midway, fill=white, font=\footnotesize, inner sep=1.5pt},
        part label/.style={font=\scriptsize, color=black, align=center}
    ]

    \begin{scope}[local bounding box=cyclic]
        \node[vertex] (C1) at (90:2.2) {$P_1$};
        \node[vertex] (C2) at (210:2.2) {$P_2$};
        \node[vertex] (C3) at (330:2.2) {$P_3$};
        \node[part label, above=0.1cm of C1] {$\{\{v_1\}, \{v_2, v_3\}\}$};
        \node[part label, below=0.2cm of C2] {$\{\{v_2\}, \{v_1, v_3\}\}$};
        \node[part label, below=0.2cm of C3] {$\{\{v_3\}, \{v_1, v_2\}\}$};
        
        \draw (C1) -- (C2) node[edge label] {$v_3$};
        \draw (C2) -- (C3) node[edge label] {$v_1$};
        \draw (C3) -- (C1) node[edge label] {$v_2$};
        
        \node at (0, -2.8) {Cyclic $T$-triangle};
    \end{scope}

    \begin{scope}[xshift=9.0cm, local bounding box=radial]
        \node[vertex] (M1) at (90:2.2) {$P_1$};
        \node[vertex] (M2) at (210:2.2) {$P_2$};
        \node[vertex] (M3) at (330:2.2) {$P_3$};
        
        \node[part label, above=0.1cm of M1] {$\{\{v_1, v_2, v_3\}, \emptyset\}$};
        \node[part label, below=0.2cm of M2] {$\{\{v_1, v_2\}, \{v_3\}\}$};
        \node[part label, below=0.2cm of M3] {$\{\{v_1, v_3\}, \{v_2\}\}$};
        
        \draw (M1) -- (M2) node[edge label] {$v_3$};
        \draw (M2) -- (M3) node[edge label] {$v_1$};
        \draw (M3) -- (M1) node[edge label] {$v_2$};
        
        \node at (0, -2.8) {Radial $T$-triangle};
    \end{scope}
    
    \end{tikzpicture}
    \caption{Examples of the two types of $T$-triangles from Definition~\ref{def:T-types}. In the \emph{cyclic} type, the partitions cycle the singleton vertex. In the \emph{radial} type, $P_1$ contains the triple $\{v_1, v_2, v_3\}$ while $P_2$ and $P_3$ break it into a pair and a singleton. 
    Lemma~\ref{lem:triple_vertex} shows these are the only types of $T$-triangles.}
    \label{fig:triangle_types} 
\end{figure}

\begin{figure}[h]
    \centering
    \begin{tikzpicture}[
        vertex/.style={circle, draw, minimum size=7mm, inner sep=0pt, font=\small},
        edge label/.style={midway, fill=white, font=\footnotesize, inner sep=1.5pt},
        spoke label/.style={midway, fill=white, font=\scriptsize, inner sep=0.5pt}
    ]

    \begin{scope}[local bounding box=split]
        \node[vertex] (C) at (0,0) {$P_0$};
        \node[vertex] (O1) at (90:2.2) {$P_1$};
        \node[vertex] (O2) at (210:2.2) {$P_2$};
        \node[vertex] (O3) at (330:2.2) {$P_3$};
        
        \draw (O1) -- (O2) node[edge label] {$v_3$};
        \draw (O2) -- (O3) node[edge label] {$v_1$};
        \draw (O3) -- (O1) node[edge label] {$v_2$};
        \draw (C) -- (O1) node[spoke label] {$v_2, v_3$};
        \draw (C) -- (O2) node[spoke label] {$v_1, v_3$};
        \draw (C) -- (O3) node[spoke label] {$v_1, v_2$};
        
        \node at (0,-2) {Split $T$-tetrahedron}; 
    \end{scope}
    
    \begin{scope}[xshift=7.0cm, local bounding box=fused]
        \node[vertex] (F) at (0,0) {$P_4$};
        \node[vertex] (R1) at (90:2.2) {$P_1$};
        \node[vertex] (R2) at (210:2.2) {$P_2$};
        \node[vertex] (R3) at (330:2.2) {$P_3$};
        
        \draw (R1) -- (R2) node[edge label] {$v_3$};
        \draw (R2) -- (R3) node[edge label] {$v_1$};
        \draw (R3) -- (R1) node[edge label] {$v_2$};
        \draw (F) -- (R1) node[spoke label] {$v_1$};
        \draw (F) -- (R2) node[spoke label] {$v_2$};
        \draw (F) -- (R3) node[spoke label] {$v_3$};
        
        \node at (0,-2) {Fused $T$-tetrahedron}; 
    \end{scope}
    \end{tikzpicture}
    \caption{Examples of the two $T$-tetrahedra from Definition~\ref{def:T-types}, corresponding to the tetrahedra in Figure~\ref{fig:K5-minus-edge} depicting $\B_3(\overline{K}_3)$. The \emph{split tetrahedron} (left) is induced by the all-singleton partition $P_0$ and the cyclic $T$-triangle $\{P_1, P_2, P_3\}$, while the \emph{fused tetrahedron}  (right) is induced by the triple partition $P_4$ and the triangle $\{P_1, P_2, P_3\}$. Theorem~\ref{thm:clique_classification}  shows these are the only types of $T$-tetrahedra.}
    \label{fig:k4_structures}
\end{figure}
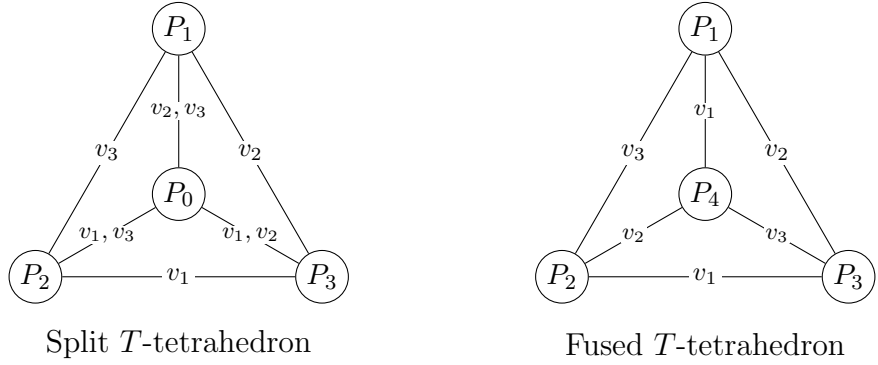

Figure~\ref{fig:K5-minus-edge} shows a Bell coloring graph $\B_3(\overline{K}_3)$ that exhibits all five types of cliques described in Definitions~\ref{def:S-vs-T}~and~\ref{def:T-types}. The vertices of the Bell coloring graph are the five 3-partitions of the set $\{v_1,v_2,v_3\}$ labeled as follows:
\begin{align*}
    P_1 &= \{\{v_1\}, \{v_2, v_3\}, \emptyset\},\qquad P_0 = \{\{v_1\}, \{v_2\}, \{v_3\}\},  \\
    P_2 &= \{\{v_2\}, \{v_1, v_3\}, \emptyset\}, \qquad
    P_4 = \{\{v_1, v_2, v_3\}, \emptyset, \emptyset\}, \\
    P_3 &= \{\{v_3\}, \{v_1, v_2\}, \emptyset\}.
\end{align*}
In this graph, $P_0P_i$ for $i=1,2,3$ is an edge, $P_4P_i$ for $i=1,2,3$ is an edge,  $P_1,P_2,P_3$ are pairwise adjacent, and $P_0P_4$ is {\em not} an edge. 

\begin{figure}[h]
\centering
\begin{tikzpicture}[
    vertex/.style={circle, draw, fill=white, minimum size=8mm, inner sep=0pt, font=\small},
    edge label/.style={midway, fill=white, font=\footnotesize, inner sep=1.5pt},
    part label/.style={font=\scriptsize, color=black, align=center},
    x={(-0.6cm, -0.4cm)}, y={(1cm, 0cm)}, z={(0cm, 1cm)},
    scale=1.2
]

\coordinate (P1) at (0, -2, 0);
\coordinate (P2) at (2, 1, 0);
\coordinate (P3) at (-2, 1, 0);

\coordinate (P4) at (0, 0, 3);

\coordinate (P0) at (0, 0, -3);

\draw (P1) -- (P2) node[edge label, pos=0.53] {$v_3$};
\draw (P0) -- (P1) node[edge label, pos=0.60] {$v_2, v_3$};
\draw (P0) -- (P2) node[edge label, pos=0.55] {$v_1, v_3$};
\draw (P0) -- (P3) node[edge label, pos=0.49] {$v_1, v_2$};

\draw (P1) -- (P3) node[edge label, pos=0.55] {$v_2$}; 
\draw (P2) -- (P3) node[edge label] {$v_1$}; 
\draw (P4) -- (P1) node[edge label, pos=0.5] {$v_1$};
\draw (P4) -- (P2) node[edge label, pos=0.4] {$v_2$};
\draw (P4) -- (P3) node[edge label, pos=0.6] {$v_3$};

\node[vertex] at (P1) {$P_1$};
\node[part label, below left=0.1cm of P1] {};

\node[vertex] at (P2) {$P_2$};
\node[part label, below right=0.1cm of P2] {};

\node[vertex] at (P3) {$P_3$};
\node[part label, above=0.1cm of P3] {};

\node[vertex] at (P4) {$P_4$};
\node[above=0.4cm of P4] {$\{\{v_1, v_2, v_3\},\emptyset,\emptyset\}$};

\node[vertex] at (P0) {$P_0$};
\node[below=0.4cm of P0] {$\{\{v_1\}, \{v_2\}, \{v_3\}\}$};

\end{tikzpicture}
\caption{Visualization of $K_5-e \cong \B_3(\overline{K}_3)$ as two tetrahedra $\{P_1, P_2, P_3, P_4\}$ and $\{P_1, P_2, P_3, P_0\}$ sharing the common base triangle $\{P_1, P_2, P_3\}$. The partitions $P_0$ and $P_4$ are the only nonadjacent pair of vertices.}
\label{fig:K5-minus-edge}
\end{figure}
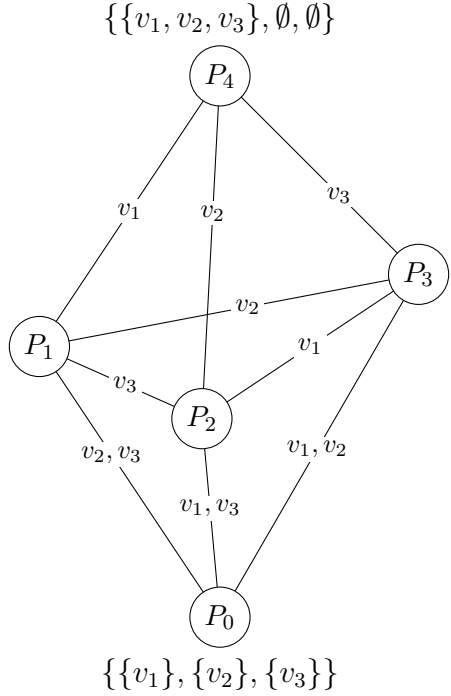

Observe that $\B_3(\overline{K}_3)$ contains seven triangles and two tetrahedra. The three triangles that include $P_0$ are $S$-triangles, each with two doubly realized edges. The three triangles that include $P_4$ are radial $T$-triangles. The triangle induced by $\{P_1,P_2,P_3\}$ is a cyclic $T$-triangle. Partitions $\{P_1,P_2,P_3,P_4\}$ induce a fused $T$-tetrahedron, and partitions $\{P_0,P_1,P_2,P_3\}$ induce a split $T$-tetrahedron. 

We now show that the particular graph $\B_3(\overline{K}_3)$ is sufficiently rich to capture the behavior of any $T$-triangle, allowing us to classify them as either cyclic or radial.

\begin{lemma}\label{lem:triple_vertex}
Every $T$-triangle in $\B_k(G)$ is either cyclic or radial.
\end{lemma}

\begin{proof}
Let $\{P_1, P_2, P_3\}$ be a $T$-triangle with edges 
\[ 
P_1 \sim_{v_3} P_2,\quad P_2 \sim_{v_1} P_3,\quad P_3 \sim_{v_2} P_1 
\] and $v_1,v_2,v_3$ distinct (see Corollary~\ref{cor:triangle_double_edge}). For each adjacency $P_i \sim_{v_\ell} P_j$, the partitions $P_i$ and $P_j$
agree on $V(G)\setminus\{v_1,v_2,v_3\}$. In particular, any two of
$P_1,P_2,P_3$ coincide outside $\{v_1,v_2,v_3\}$, so there is a common restriction $R$ on $V(G)\setminus\{v_1,v_2,v_3\}$ such that
\[
P_i = P_i|_{\{v_1,v_2,v_3\}} \cup R \quad\text{for } i=1,2,3.
\]
Thus, the entire structure of the triangle is encoded by the three restrictions
$P_i|_{\{v_1,v_2,v_3\}}$. These correspond to three vertices of the Bell coloring graph $\B_3(\overline{K}_3)$, which is isomorphic to $K_5-e$ (see Figure~\ref{fig:K5-minus-edge}). In that graph, there are exactly three types of triangles:
\begin{itemize}
    \item a base triangle whose vertices are the three
    ``pair+singleton'' partitions of $\{v_1,v_2,v_3\}$;
    \item triangles that include the ``triple" partition $\{\{v_1,v_2,v_3\},\emptyset,\emptyset\}$;
    \item triangles that include the ``all-singleton" partition $\{\{v_1\},\{v_2\},\{v_3\}\}$.
\end{itemize}

Edges incident to the all-singleton partition are doubly realized (they split a pair into two singletons), so by Corollary~\ref{cor:triangle_double_edge} any triangle involving the all-singleton partition is an $S$-triangle, not a $T$-triangle. Since we started with a $T$-triangle in $\B_k(G)$, its restriction to $\{v_1,v_2,v_3\}$ must correspond to one of the first two types. Translating these two possibilities back to $G$ yields:

\smallskip\noindent
\emph{Cyclic form.} The restrictions $P_i|_{\{v_1,v_2,v_3\}}$ are exactly the three
``pair+singleton'' partitions:
\[
\{\{v_1\},\{v_2,v_3\}\},\quad
\{\{v_2\},\{v_1,v_3\}\},\quad
\{\{v_3\},\{v_1,v_2\}\}.
\]
In this case, each $P_i$ contains $\{v_i\}$ as a singleton part, and the edges
are realized by $v_1,v_2,v_3$ as above, so $\{P_1,P_2,P_3\}$ is a cyclic
$T$-triangle in the sense of Definition~\ref{def:T-types}.

\smallskip\noindent
\emph{Radial form.} The restrictions $P_i|_{\{v_1,v_2,v_3\}}$ consist of one ``triple'' and two
``pair+singleton'' partitions:
\[
\{\{v_1,v_2,v_3\},\emptyset\},\quad \{\{v_1,v_2\},\{v_3\}\},\quad \{\{v_1,v_3\},\{v_2\}\}.
\]
Here, $P_1$ contains the triple $\{v_1,v_2,v_3\}$ as a part, and again the edges are realized by $v_1,v_2,v_3$, so $\{P_1,P_2,P_3\}$ is a radial $T$-triangle in the sense of Definition~\ref{def:T-types}. \end{proof}

The next lemma shows that once a $T$-triangle is fixed, any vertex adjacent to all three triangle vertices has a partition that is already determined on all vertices outside $\{v_1,v_2,v_3\}$. This rigidity will be crucial for enumerating $T$-tetrahedra and for ruling out the possibility of larger $T$-cliques.

\begin{lemma}\label{lem:fourth_vertex}
Let $\{P_1,P_2,P_3\}$ be a $T$-triangle in $\B_k(G)$ with realizers $v_1,v_2,v_3$. Let $W \colonequals V(G)\setminus\{v_1,v_2,v_3\}$, and let $R$ be the common restriction of the $P_i$ to $W$.
If $P$ is a vertex of $\B_k(G)$ adjacent to each $P_i$, then $P|_W = R$.
\end{lemma}

\begin{proof}
Suppose, to the contrary, that $P|_W \neq R$. For each $i$, write $P \sim_{a_i} P_i$ for some $a_i\in V(G)$.
If $a_i \in \{v_1, v_2, v_3\}$, then $P$ and $P_i$ agree on $W$; since $P_i|_W = R$, this implies $P|_W = R$, a contradiction.
Therefore, we must have $a_i \in W$ for all $i=1,2,3$.

Since $a_i \in W$, the adjacency $P \sim_{a_i} P_i$ forces $P$ and $P_i$ to
agree on $V(G)\setminus W = \{v_1,v_2,v_3\}$. Thus, we obtain
\[
P_1|_{\{v_1,v_2,v_3\}} = P|_{\{v_1,v_2,v_3\}} = P_2|_{\{v_1,v_2,v_3\}}.
\]
However, by Lemma~\ref{lem:triple_vertex}, the $T$-triangle
$\{P_1,P_2,P_3\}$ is either cyclic or radial, and in both patterns the three restrictions $P_i|_{\{v_1,v_2,v_3\}}$ are pairwise distinct (as is clear from the explicit forms in its proof). This contradiction shows that our assumption $P|_W \neq R$ was false, so $P|_W = R$.
\end{proof}

We now state and prove the general classification theorem. 

\begin{theorem}\label{thm:clique_classification}
For any clique $\mathcal K=\{P_1,\dots,P_m\}$ with $m\ge 1$ in $\B_k(G)$, exactly one of the following holds:
\begin{enumerate}
    \item $\mathcal K=\{P_1,P_2,P_3\}$ is a cyclic $T$-triangle;
    \item $\mathcal K=\{P_1,P_2,P_3\}$  is a radial $T$-triangle;
    \item $\mathcal K=\{P_1,P_2,P_3,P_4\}$  is a split $T$-tetrahedron; 
    \item $\mathcal K=\{P_1,P_2,P_3,P_4\}$ is a fused $T$-tetrahedron; or 
    \item $\mathcal K=\{P_1,\dots,P_m\}$ is an $S$-clique.
\end{enumerate} 
\end{theorem}

\begin{proof}[Proof of Theorem~\ref{thm:clique_classification}] 
Let $\mathcal{K}=\{P_1, \dots, P_m\}$ be a clique in $\B_k(G)$. Let $V_{ij} = \{v \in V(G) \mid P_i \sim_v P_j\}$. By Lemma~\ref{lem:double_edge}, we have $|V_{ij}| \in \{1, 2\}$. We proceed by a case analysis on $m$.

\smallskip 

\noindent \textbf{Case $m=1$ (vertex).} A single vertex vacuously has an anchor, so $\mathcal K$ is an $S$-clique.

\noindent \textbf{Case $m=2$ (edge).} 
The unique edge in $\mathcal{K}$ is either singly realized or doubly realized (see Lemma~\ref{lem:double_edge} for the latter). In either case, $\mathcal{K}$ is an $S$-clique.

\smallskip 

\noindent \textbf{Case $m=3$ (triangle).}
Apply Lemma~\ref{lem:triangle_realizers} to $\{P_1,P_2,P_3\}$. Either there is an anchor $u$ ($S$-clique) or each edge is singly realized by distinct vertices ($T$-triangle).

\smallskip 

\noindent \textbf{Case $m=4$ (tetrahedron).} If all faces are $S$-triangles, the clique is an $S$-clique; indeed, by Lemma~\ref{lem:anchor_uniqueness}, $S$-triangles sharing an edge must share the same anchor.

Otherwise, $\mathcal{K}$ contains a $T$-triangle $\{P_1, P_2, P_3\}$ realized by distinct vertices $\{v_1, v_2, v_3\}$. By Lemma~\ref{lem:triple_vertex},
all partitions in this triangle coincide on $W \colonequals V(G) \setminus \{v_1, v_2, v_3\}$. Any fourth vertex $P \in \mathcal{K}$ must also agree with them on $W$ by Lemma~\ref{lem:fourth_vertex}. Thus, the structure of $\mathcal{K}$ is determined entirely by the restrictions of its vertices to $\{v_1, v_2, v_3\}$.

The graph of stable partitions on three vertices is isomorphic to $\B_3(\overline{K}_3) \cong K_5 - e$ (see Figure~\ref{fig:K5-minus-edge}). The only $4$-cliques in $K_5 - e$ are the two tetrahedra sharing the base triangle. These correspond precisely to the split $K_4$ (containing the all-singleton partition $P_0$) and the fused $K_4$ (containing the triple partition $P_4$).

\smallskip 

\noindent\textbf{Case $m \ge 5$ (higher order).}
If $\mathcal{K}$ contains a $T$-triangle, it must embed into the local graph described above. Since $\B_3(\overline{K}_3) \cong K_5-e$ contains no 5-clique, no $T$-clique can extend to size $m \ge 5$. Thus, every $K_4$ subclique of $\mathcal{K}$ must be an $S$-clique. By Lemma~\ref{lem:anchor_uniqueness}, the anchors of overlapping $S$-cliques must coincide, so the entire clique $\mathcal{K}$ is an $S$-clique. \end{proof}

\subsection{Forbidden induced subgraphs}\label{sec:forbidden}

As a first application of the clique classification, we identify a small graph that cannot occur as a Bell coloring graph.

\begin{theorem}\label{thm:K4-e}
The diamond graph $K_4-e$ is not isomorphic to $\B_k(G)$ for any graph $G$ and integer $k$.
\end{theorem}

\begin{proof}
Suppose $\B_k(G) \cong K_4-e$ with vertices $\{P_1, P_2, P_3, P_4\}$ and missing edge $P_1P_4$. The graph consists of two triangles $\triangle_1 = P_1P_2P_3$ and $\triangle_2 = P_2P_3P_4$ sharing the edge $P_2P_3$.

First, $\triangle_1$ is not a $T$-triangle. If it were, its vertices would arise from three distinct realizers $\{v_1, v_2, v_3\}$. As shown in Figure~\ref{fig:K5-minus-edge}, the local structure of partitions on three vertices is isomorphic to $K_5-e$. In that graph, every triangle is contained in a $K_4$. Thus, the vertices $\{P_1, P_2, P_3\}$ would extend to a $K_4$ involving $P_4$. This forces $P_1 \sim P_4$, a contradiction. By symmetry, $\triangle_2$ is also not a $T$-triangle.

Therefore, both $\triangle_1$ and $\triangle_2$ are $S$-triangles. Let $u$ and $v$ be their respective anchors. If $u=v$, transitivity of adjacency at $u$ implies $P_1 \sim_u P_4$, a contradiction. Thus $u \neq v$, and the shared edge $P_2P_3$ is realized by both $u$ and $v$. By Lemma~\ref{lem:double_edge}, we may write:
\[
P_2 = \{\{u, v\}, \emptyset\} \cup R \qquad \text{and} \qquad P_3 = \{\{u\}, \{v\}\} \cup R.
\]
Since $P_1\neq P_3$, the partition $P_1$ must be obtained from $P_2$ by moving $u$ to a part $A \in R$. Similarly, $P_4$ is obtained from $P_2$ by moving $v$ to a part $B \in R$. Define a new partition $P_5$ by performing both moves simultaneously:
\[
P_5 \colonequals (P_2 \setminus \{\{u,v\}, A, B\}) \cup \{A \cup \{u\}, B \cup \{v\}, \emptyset\}.
\]
(If $A=B$, the new parts merge into $A \cup \{u, v\}$). $P_5$ is stable because $A \cup \{u\}$ is stable (implied by $P_1$) and $B \cup \{v\}$ is stable (implied by $P_4$). Since $P_5$ differs from $P_1, \dots, P_4$ by the placement of at least one vertex, its existence implies $|\B_k(G)| \ge 5$, a contradiction.
\end{proof}

Despite not being a Bell coloring graph, $K_4-e$ appears as an induced subgraph of a Bell coloring graph. Indeed, $K_4-e$ is an induced subgraph of $K_5-e$, which is a Bell coloring graph due to $K_5-e\cong \B_3(\overline{K}_3)$ from Figure~\ref{fig:K5-minus-edge}. On the other hand, our classification theorem shows that $K_6-e$ is a forbidden subgraph. 

\begin{theorem}\label{thm:forbidden}
The graph $K_{6}-e$ is not an induced subgraph of any Bell coloring graph. Hence, the set of graphs $K_n-e$ for $n\ge 6$ is an infinite family of forbidden induced subgraphs.
\end{theorem}

\begin{proof}
    Suppose $H \cong K_6-e$ is an induced subgraph of $\B_k(G)$. The graph $H$ consists of two copies of $K_5$, say $A$ and $B$, intersecting in a $K_4$.
    By Theorem~\ref{thm:clique_classification}, every clique of size $\ge 5$ is an $S$-clique. Thus, $A$ and $B$ are $S$-cliques.
    Since $A$ and $B$ intersect in a $K_4$, they share an $S$-triangle. By Lemma~\ref{lem:anchor_uniqueness}, $A$ and $B$ share a common anchor vertex $u$. Consequently, the union $H = A \cup B$ is entirely realized by $u$, implying $H$ is a complete graph $K_6$. This contradicts the missing edge in $K_6-e$.

    Thus, $K_6-e$ is not an induced subgraph of any Bell coloring graph. Since $K_n-e$ contains $K_6-e$ as an induced subgraph, it follows that $K_n-e$ is also forbidden for each $n\geq 6$.
\end{proof}

Similar reasoning forbids other graphs composed of intersecting cliques. For instance, $K_7-K_{2,2}$ (formed by two copies of $K_5$ sharing a triangle) is not an induced subgraph of any Bell coloring graph. While these examples suggest a rich theory of forbidden subgraphs for $\B_k(G)$, a complete classification lies beyond the scope of this paper.

\section{Matching graphs as a source of Bell coloring graphs}~\label{sec:matchings} 

In this section, we introduce the matching reconfiguration graph and establish its connection to Bell coloring graphs. As a consequence, we obtain Theorem~\ref{thm:trees-cycles}, stating that all trees and cycle graphs are realizable as Bell coloring graphs.

Recall that a \emph{matching} is a set of edges that are pairwise vertex-disjoint. For a matching $M \subseteq E(G)$, we let $G\langle M \rangle$ denote the spanning subgraph of $G$ with edge set precisely $M$. In other words, $G\langle M\rangle$ is obtained from $G$ by removing all edges in $E[G]\setminus M$, while preserving the vertex set of $G$.

\begin{defn}
Given a graph $G$ and $k\in\mathbb{N}$, the \emph{matching reconfiguration graph} $\mathcal{M}_k(G)$ is the graph whose vertex set consists of all matchings in $G$ of size at least $k$. Two matchings $M_1, M_2$ are adjacent if there exists $v\in V(G)$ such that $G\langle M_1\rangle-v = G\langle M_2\rangle-v$. 
\end{defn}

We write $M_1 \sim_v M_2$ to indicate adjacency via $v$, which implies that $M_1$ and $M_2$ differ only by edges incident to $v$. We write $M_1 \sim M_2$ if $M_1 \sim_v M_2$ for some $v\in V(G)$. 

We show that the matching graph $\mathcal{M}_k(G)$ has a Bell coloring graph interpretation. Figure~\ref{fig:M-cong-B} gives a concrete example, and Proposition~\ref{prop:matchings} establishes the general result.

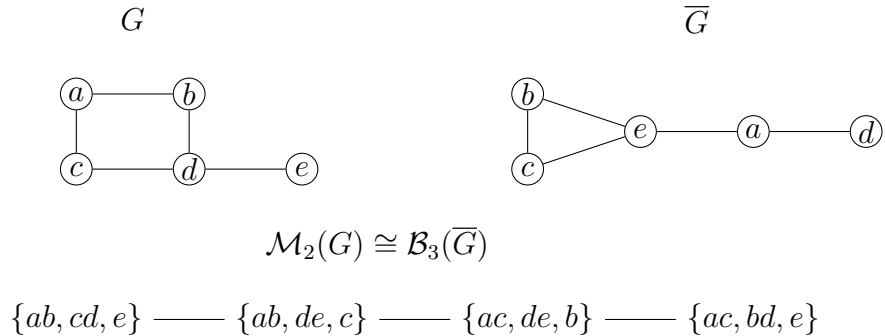
\begin{figure}[h]
\centering
\begin{tikzpicture}
 \tikzset{
    vertex/.style={
      circle,draw,
      inner sep=0pt,
      minimum size=12pt  
    },
  }
  \node[vertex] (a) at (0,1) {$a$};
  \node[vertex] (b) at (1.5,1) {$b$};
  \node[vertex] (c) at (0,0) {$c$};
  \node[vertex] (d) at (1.5,0) {$d$};
  \node[vertex] (e) at (3,0) {$e$};

  \draw (a)--(b)--(d)--(c)--(a);
  \draw (d)--(e);

  \node at (0.75,2) {$G$};

  \begin{scope}[xshift=6cm]
    \node[vertex] (b2) at (0,1) {$b$};
    \node[vertex] (c2) at (0,0) {$c$};
    \node[vertex] (e2) at (1.5,0.5) {$e$};
    \node[vertex] (a2) at (3,0.5) {$a$};
    \node[vertex] (d2) at (4.5,0.5) {$d$};

    \draw (b2)--(c2);
    \draw (b2)--(e2);
    \draw (c2)--(e2);
    \draw (e2)--(a2)--(d2);

    \node at (2.25,2) {$\overline{G}$};
  \end{scope}

  \node at (4,-1.0) {$\mathcal{M}_2(G) \cong \mathcal{B}_3(\overline{G})$};

  \node (x1) at (0,-2)   {$\{ab,cd,e\}$};
  \node (x2) at (3,-2)   {$\{ab,de,c\}$};
  \node (x3) at (6,-2)   {$\{ac,de,b\}$};
  \node (x4) at (9,-2)   {$\{ac,bd,e\}$};

  \draw (x1)--(x2)--(x3)--(x4);
\end{tikzpicture} 
\caption{Illustration that $\mathcal M_2(G)\cong \B_3(\overline{G})$ for a particular graph $G$ on 5 vertices. Vertex labels can be interpreted as matchings of size at least 2 or stable 3-partitions.}
\label{fig:M-cong-B}
\end{figure}

\begin{prop}\label{prop:matchings}
If $G$ is a triangle-free graph of order $n$, then $\mathcal{M}_k(G) \cong \B_{n-k}(\overline{G})$.
\end{prop}
 
\begin{proof}

We construct a graph isomorphism between $\mathcal{M}_k(G)$ and $\B_{n-k}(\Gbar)$. Given a matching $M\subseteq E(G)$, we define a stable partition $\Phi(M)$ of $\Gbar$. Suppose $M$ contains edges $u_1 v_1, \dots, u_{\ell} v_{\ell}$ and isolated vertices $w_{1}, \dots, w_{n-2\ell}$ where $\ell\geq k$. Let
\[
\Phi(M) = \{\{u_1, v_1\}, \dots, \{u_{\ell}, v_{\ell}\}, \{w_1\}, \dots, \{w_{n-2\ell}\}\}.
\]
For each $1\leq i\leq \ell$, the set $\{u_i, v_i\}$ is indeed independent in $\Gbar$. The map $\Phi$ is injective because a stable partition of the above form uniquely determines the underlying matching. Moreover, $\Phi$ is surjective because $\Gbar$ contains no independent set of size three (equivalently, $G$ is triangle-free), so every stable $(n-k)$-partition of the indicated form comes from a matching.

Finally, note that $M_1$ and $M_2$ are adjacent matchings if $G\langle M_1\rangle-v = G\langle M_2\rangle-v$. This is equivalent to $\Phi(M_1)-v = \Phi(M_2) - v$. We deduce that $\Phi$ is a graph isomorphism.
\end{proof}

The isomorphism in Proposition~\ref{prop:matchings} enables us to realize all cycles and trees as Bell coloring graphs. We focus on the case where $n=|V(G)|=2k+1$. More formally, if $n=2k+1$ then we call $\mathcal{M}_k(G)$ the \emph{near-perfect matching graph} of $G$. The following lemma gives three equivalent conditions for two near-perfect matchings to be adjacent in a matching graph, which helps us reason about matching graphs that are cycles. In particular, it shows that near-perfect matchings are adjacent if and only if the unmatched vertices $u$ and $v$ have a common neighbor $w$ such that the symmetric difference of the matching is $\{uw,uv\}$. 

\begin{lemma}\label{lem:dist-2}
    Let $M_1,M_2$ be two near-perfect matchings of $G$ with $v$ unmatched in $M_1$ and $u$ unmatched in $M_2$. The following are equivalent:
    \begin{enumerate}
        \item\label{item:adj1} $G\zap{M_1}-w=G\zap{M_2}-w$ for some $w\in V(G)$.
        \item\label{item:adj2} $M_1=(M_2\setminus\{vw\})\cup\{uw\}$ for some $w\in V(G)$.
        \item\label{item:adj3} $uw\in M_1$, $vw\in M_2$, and $G\zap{M_1}-u-v-w=G\zap{M_2}-u-v-w$ for some $w\in V(G)$.
    \end{enumerate}
\end{lemma}
\begin{proof}
    The desired equivalence follows from three separate implications. 
     
     (\ref{item:adj1})$\implies$(\ref{item:adj3}). Suppose $G\zap{M_1}-w=G\zap{M_2}-w$. Then certainly $G\zap{M_1}-u-v-w=G\zap{M_2}-u-v-w$. This implies that any edges in $M_1$, but not in $M_2$, must be between the vertices $u,v,w$. Since only one such edge can be present in any valid matching, there must be $k-1$ edges present in $G\zap{M_1}-u-v-w$. Since $G\zap{M_1}-u-v-w$ contains only $2k-2$ vertices, there must not be any edges in $M_1$ or $M_2$ with one endpoint in $\{u,v,w\}$ and another in $G\setminus \{u,v,w\}$. Therefore, both $M_1$ and $M_2$ consist of $k-1$ common edges with both endpoints in $G\setminus \{u,v,w\}$, and one edge each in the set $\{uv,uw,vw\}$. If $uv\in M_1$, then $G\zap{M_1}-w$ has $k$ edges, which would fix $M_1=M_2$, a contradiction. Therefore, by the distinctness of $M_1$ and $M_2$, we must have that $uw\in M_1$ and $vw\in M_2$, up to relabeling.
        
         (\ref{item:adj3})$\implies$(\ref{item:adj2}). Suppose that $uw\in M_1$, $vw\in M_2$, and $G\zap{M_1}-u-v-w=G\zap{M_2}-u-v-w$. This means that $M_1=E(G\zap{M_1}-u-v-w)\cup\{uw\}$ and $M_2=E(G\zap{M_2}-u-v-w)\cup\{vw\}$. Therefore, $(M_2\setminus\{vw\})\cup\{uw\}=E(G\zap{M_2}-u-v-w)\cup\{uw\}=M_1$.
        
         (\ref{item:adj2})$\implies$(\ref{item:adj1}). Suppose $M_1=(M_2\setminus\{vw\})\cup\{uw\}$, then $M_1\setminus \{uw\}=M_2\setminus\{vw\}$. This implies that $G\zap{M_1}-w=G\zap{M_1\setminus \{uw\}}-w=G\zap{M_2\setminus\{vw\}}-w=G\zap{M_2}-w$, as desired.
\end{proof}

We are now ready to prove that every odd cycle is a Bell coloring graph. Figure~\ref{fig:odd-cycle} illustrates the proof of Lemma~\ref{lem:odd-cycles}, which establishes that the near-perfect matching graph of an odd cycle is the odd cycle and, due to Proposition~\ref{prop:matchings}, is realizable as a Bell coloring graph.

\begin{lemma}\label{lem:odd-cycles}
Every odd cycle is a Bell coloring graph. In particular, $\B_{k+1}(\overline{C_{2k+1}})\cong \mathcal M_k(C_{2k+1})\cong C_{2k+1}$ for any $k\ge 1$.
\end{lemma}

\begin{proof}
Consider the near-perfect matching graph $\mathcal{M}_k(C_{n})$ for $n=2k+1$. If we select a vertex $v$ to be the isolated vertex, then the remaining graph is the even path graph $P_{2k}$. There is only one possible  set of $k$ edges of $P_{2k}$ that remains vertex-disjoint. Therefore, we can label a given vertex $M_v$ of $\mathcal{M}_k(C_n)$ according to the vertex $v$ which is unmatched. By Lemma~\ref{lem:dist-2}, $M_u$ can only be adjacent to $M_v$ if $u$ and $v$ share a neighbor, $w$. If this is the case, then $M_u=(M_v\setminus\{uw\})\cup\{vw\}$, so $M_u$ and $M_v$ will indeed be adjacent whenever $u$ and $v$ share a neighbor. If we label the vertices of $C_n$ as $v_1,\dots,v_n$, then we will have adjacencies $M_{v_i}M_{v_{i+2}}$ where $i+2$ is taken modulo $n$. Since $2$ is coprime to $n$, this generates $C_n$. Therefore $\mathcal{M}_k(C_{2k+1})\cong C_{2k+1}$.
\end{proof}

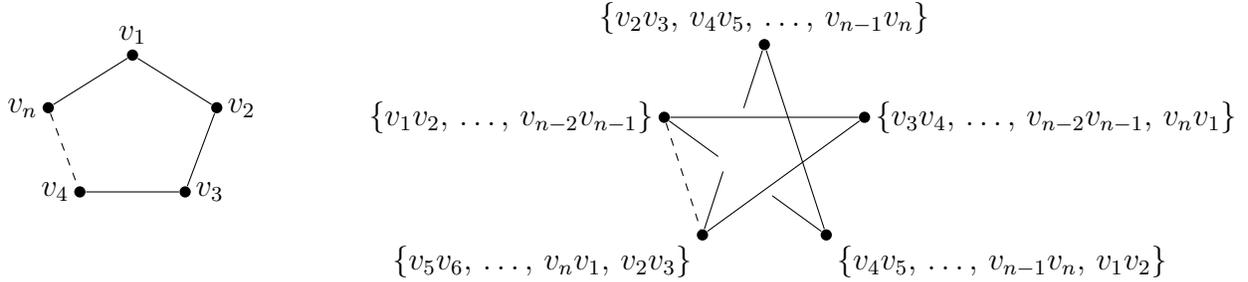
\begin{figure}[h]
\centering
\begin{tikzpicture}[scale=0.7]
  \tikzset{
    vertex/.style={circle,fill=black,inner sep=1.5pt},
  }

  \node[vertex] (v1) at (0,1.8) {};
  \node[vertex] (v2) at (1.6,0.8) {};
  \node[vertex] (v3) at (1.0,-0.8) {};
  \node[vertex] (v4) at (-1.0,-0.8) {};
  \node[vertex] (vn) at (-1.6,0.8) {};

  \node[above]  at (v1) {$v_1$};
  \node[right]  at (v2) {$v_2$};
  \node[right]  at (v3) {$v_3$};
  \node[left]   at (v4) {$v_4$};
  \node[left]   at (vn) {$v_n$};

  \draw (v1)--(v2)--(v3)--(v4);
  \draw (v1)--(vn);
  \draw[dashed] (vn)--(v4);

  \begin{scope}[xshift=12cm]
    \node[vertex] (b) at (90:2)   {}; 
    \node[vertex] (c) at (18:2)   {};   
    \node[vertex] (d) at (-54:2)  {};  
    \node[vertex] (e) at (-126:2) {};  
    \node[vertex] (a) at (162:2)  {}; 
    \draw (a)--(c);
    \draw (c)--(e);
    \draw (e) -- ($(e)!1/3!(b)$);
    \draw ($(e)!2/3!(b)$) -- (b);
    \draw (b)--(d);
    \draw (d) -- ($(d)!1/3!(a)$);
    \draw ($(d)!2/3!(a)$) -- (a);

    \draw[dashed] (a)--(e);

    \node[above]      at (b)
      {$\{v_2v_3,\,v_4v_5,\,\ldots,\,v_{n-1}v_n\}$};

    \node[left]       at (a)
      {$\{v_1v_2,\,\ldots,\,v_{n-2}v_{n-1}\}$};

    \node[right]      at (c)
      {$\{v_3v_4,\,\ldots,\,v_{n-2}v_{n-1},\,v_nv_1\}$};

    \node[below right] at (d)
      {$\{v_4v_5,\,\ldots,\,v_{n-1}v_n,\,v_1v_2\}$};

    \node[below left] at (e)
      {$\{v_5v_6,\,\ldots,\,v_nv_1,\,v_2v_3\}$};
  \end{scope}
\end{tikzpicture}
\caption{Vertex labeling for an odd cycle $G=C_{2k+1}$ and its corresponding near-perfect matching graph illustrating that $\mathcal M_k(C_{2k+1})\cong C_{2k+1}$.}
\label{fig:odd-cycle}
\end{figure}

We are now ready to show that every even cycle is also realizable as a Bell coloring graph. To do this, we define a simple family of graphs that each consists of an even cycle with an ear of length two attached to two cycle vertices with a common neighbor. The $n=6$ example is depicted in Figure~\ref{fig:G6}.

\begin{defn}\label{def:ear-graph}
For $n=2k,k\ge 2$, we let $G_n$ be the graph on $n+1$ vertices with $V(G_n)=\{u,w,v_1,\dots,v_{n-1}\}$ and $E(G_n)=\{v_iv_{i+1}\mid 1\le i\le n-2\}\cup \{uv_1,wv_1,uv_{n-1},wv_{n-1}\}$.     
\end{defn} 

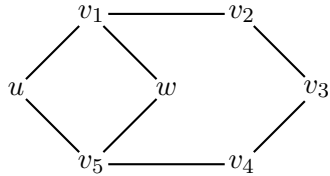
\begin{figure}[h]
    \centering
    \begin{tikzpicture}[
    thick,
    vertex/.style={fill=white, inner sep=1pt, font=\small},
    highlight/.style={line width=7pt, orange!40, cap=round, on background layer},
    graph label/.style={scale=1, anchor=west},
    scale=1 ,rotate=90]
]
      \node[vertex] (0) at (0, 2) {$u$};
      \node[vertex] (1) at (0, 0) {$w$};
      \node[vertex] (2) at (0, -2) {$v_3$};
      \node[vertex] (3) at (1, -1) {$v_2$};
      \node[vertex] (4) at (1, 1) {$v_1$};
      \node[vertex] (5) at (-1, 1) {$v_5$};
      \node[vertex] (6) at (-1, -1) {$v_4$};
      \draw[edge] (3) -- (2);
      \draw[edge] (4) -- (0);
      \draw[edge] (4) -- (1);
      \draw[edge] (4) -- (3);
      \draw[edge] (5) -- (0);
      \draw[edge] (5) -- (1);
      \draw[edge] (6) -- (2);
      \draw[edge] (6) -- (5);
    \end{tikzpicture}
    \caption{The graph $G_6$ as in Definition~\ref{def:ear-graph}. Lemma~\ref{lem:even-cycles} shows that, in particular, the near-perfect matching graph of $G_6$ is $C_8$.}
    \label{fig:G6}
\end{figure}

\begin{lemma}\label{lem:even-cycles}
    Every even cycle is a Bell coloring graph. In particular, $\mathcal M_{k}(G_{2k})\cong C_{2k+2}$ for $k\ge 2$ and $G_{2k}$ as in Definition~\ref{def:ear-graph}.
\end{lemma}

\begin{proof} We first observe that the special case of $C_4$ is realized by $\B_2(K_2\sqcup \overline{K_2})\cong \mathcal{C}_2(\overline{K_2})\cong C_4$, an application of the constructive result that any coloring graph is realizable as a Bell coloring graph \cite{FM14}*{Proposition 2.5}. It suffices to show that $\mathcal M_{n/2}(G_{n})\cong C_{n+2}$ for all even $n\ge 2$ and $G_n$ as in Definition~\ref{def:ear-graph}. 

We start by enumerating the vertices of $\mathcal M_{n/2}(G_{n})$ for even $n\ge 2$. Each $n/2$-matching must have a unique unmatched vertex. For $i=1,3,\dots,n-1$, we note that $G_n-v_i$ has a unique bipartition with parts of size $n/2+1$ and $n/2-1$, so $v_i$ cannot be the uniquely unmatched vertex. Any other $v$ is such that $G-v$ has exactly two perfect matchings, so the $n+2$ vertices of $\mathcal M_{n/2}(G_{n})$ can be described as follows:
\begin{align*}
    V(\mathcal M_{n/2}(G_{n})) &= \{M_u,M_u',M_w,M_w'\}\cup \bigcup_{i=2,4,\dots,n-2}\{M_i, M_i'\}\text{ for } \\
    M_u &\colonequals \{v_1v_2,\dots,v_{n-3}v_{n-2},v_{n-1}w\}, \\
    M_u' &\colonequals \{wv_1,v_2v_3,\dots,v_{n-2}v_{n-1}\}, \\
    M_w &\colonequals \{uv_1,v_2v_3,\dots,v_{n-2}v_{n-1}\}, \\
    M_w' &\colonequals \{v_1v_2,\dots,v_{n-3}v_{n-2},v_{n-1}u\}, \\
    M_i &\colonequals \{uv_1,v_2v_3,\dots,v_{i-2}v_{i-1},v_{i+1}v_{i+2},\dots,v_{n-2}v_{n-1},v_{n-1}w\}\text{ for $i=2,4,\dots,n-2$,} \\
    M_i' &\colonequals \{wv_1,v_2v_3,\dots,v_{i-2}v_{i-1},v_{i+1}v_{i+2},\dots,v_{n-2}v_{n-1},v_{n-1}u\}\text{ for $i=2,4,\dots,n-2$.}
\end{align*}
Some of these matchings are depicted in Figure~\ref{fig:matchings-even-cycle}.


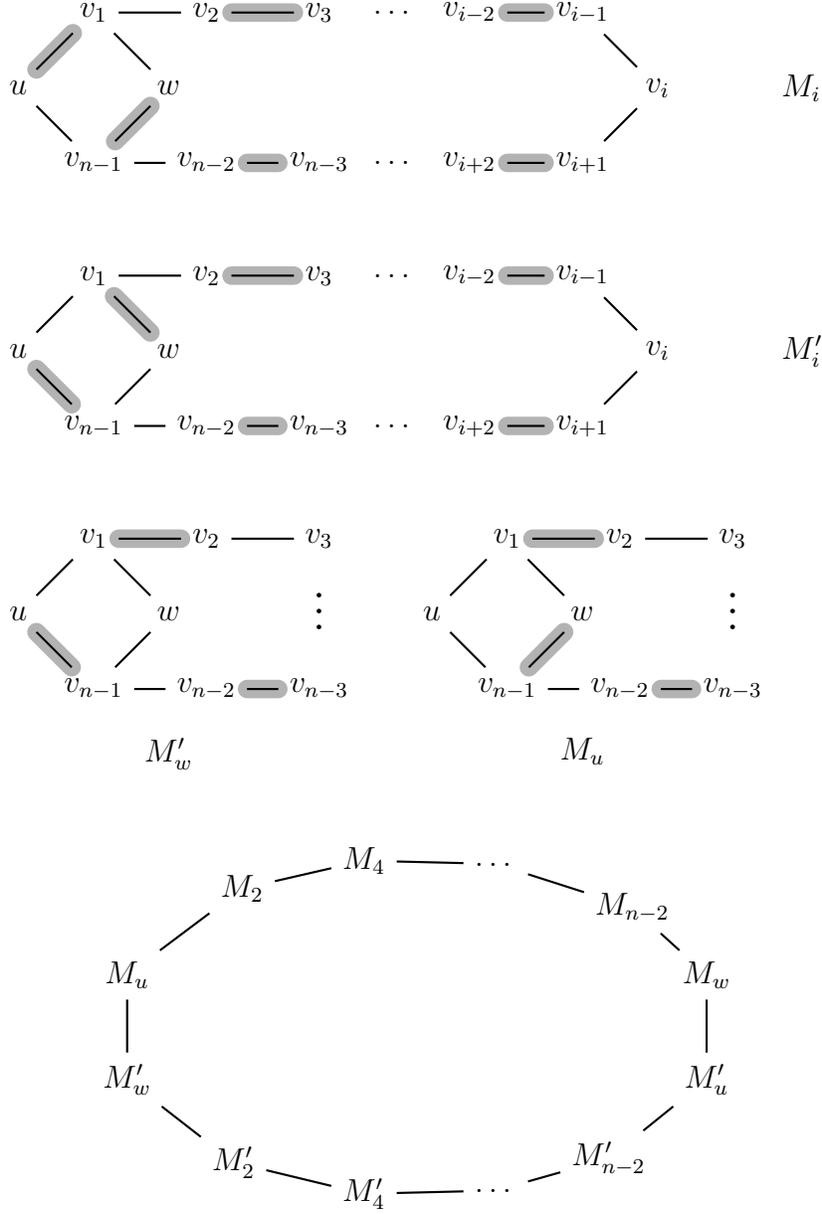
\begin{figure}
\centering 
\begin{tikzpicture}[
    thick,
    vertex/.style={fill=white, inner sep=1pt, font=\small},
    highlight/.style={line width=7pt, black!30, cap=round, on background layer},
    graph label/.style={scale=1, anchor=west},
    scale=1 
]

    \begin{scope}[shift={(0,0)}]
        \node (u) at (0,0) {$u$};
        \node (v1) at (1,1) {$v_1$};
        \node (vn1) at (1,-1) {$v_{n-1}$};
        \node (w) at (2,0) {$w$};
        \node (v2) at (2.5,1) {$v_2$};
        \node (v3) at (4,1) {$v_3$};
        \node (dots_top) at (5,1) {$\dots$};
        \node (vi-2) at (6,1) {$v_{i-2}$};
        \node (vi-1) at (7.5,1) {$v_{i-1}$};
        \node (vi) at (8.5,0) {$v_i$};
        \node (vn2) at (2.5,-1) {$v_{n-2}$};
        \node (vn3) at (4,-1) {$v_{n-3}$};
        \node (dots_bot) at (5,-1) {$\dots$};
        \node (vi+2) at (6,-1) {$v_{i+2}$};
        \node (vi+1) at (7.5,-1) {$v_{i+1}$};

        \draw[highlight] (u) -- (v1);
        \draw[highlight] (w) -- (vn1);
        \draw[highlight] (v2) -- (v3);
        \draw[highlight] (vn2) -- (vn3);
        \draw[highlight] (vi-2) -- (vi-1);
        \draw[highlight] (vi+2) -- (vi+1);

        \draw (u) -- (v1) -- (w) -- (vn1) -- (u);
        \draw (v1) -- (v2) -- (v3);
        \draw (vi-2) -- (vi-1) -- (vi);
        \draw (vn1) -- (vn2) -- (vn3);
        \draw (vi+2) -- (vi+1) -- (vi);
        
        \node[graph label] at (10,0) {$M_i$};
    \end{scope}

    \begin{scope}[shift={(0,-3.5)}]
        \node (u) at (0,0) {$u$};
        \node (v1) at (1,1) {$v_1$};
        \node (vn1) at (1,-1) {$v_{n-1}$};
        \node (w) at (2,0) {$w$};
        \node (v2) at (2.5,1) {$v_2$};
        \node (v3) at (4,1) {$v_3$};
        \node (dots_top) at (5,1) {$\dots$};
        \node (vi-2) at (6,1) {$v_{i-2}$};
        \node (vi-1) at (7.5,1) {$v_{i-1}$};
        \node (vi) at (8.5,0) {$v_i$};
        \node (vn2) at (2.5,-1) {$v_{n-2}$};
        \node (vn3) at (4,-1) {$v_{n-3}$};
        \node (dots_bot) at (5,-1) {$\dots$};
        \node (vi+2) at (6,-1) {$v_{i+2}$};
        \node (vi+1) at (7.5,-1) {$v_{i+1}$};

        \draw[highlight] (u) -- (vn1);
        \draw[highlight] (v1) -- (w);
        \draw[highlight] (v2) -- (v3);
        \draw[highlight] (vn2) -- (vn3);
        \draw[highlight] (vi-2) -- (vi-1);
        \draw[highlight] (vi+2) -- (vi+1);

        \draw (u) -- (v1) -- (w) -- (vn1) -- (u);
        \draw (v1) -- (v2) -- (v3);
        \draw (vi-2) -- (vi-1) -- (vi);
        \draw (vn1) -- (vn2) -- (vn3);
        \draw (vi+2) -- (vi+1) -- (vi);

        \node[graph label] at (10,0) {$M_i'$};
    \end{scope}
    
    \begin{scope}[shift={(0,-7.0)}]
        \node (u) at (0,0) {$u$};
        \node (v1) at (1,1) {$v_1$};
        \node (vn1) at (1,-1) {$v_{n-1}$};
        \node (w) at (2,0) {$w$};
        \node (v2) at (2.5,1) {$v_2$};
        \node (v3) at (4,1) {$v_3$};
        \node (vn2) at (2.5,-1) {$v_{n-2}$};
        \node (vn3) at (4,-1) {$v_{n-3}$};
        \node[scale=1.5] at (4,0.2) {$\vdots$};

        \draw[highlight] (u) -- (vn1);
        \draw[highlight] (v1) -- (v2);
        \draw[highlight] (vn2) -- (vn3);

        \draw (u) -- (v1) -- (w) -- (vn1) -- (u);
        \draw (v1) -- (v2) -- (v3);
        \draw (vn1) -- (vn2) -- (vn3);
        
        \node[anchor=north] at (2,-1.5) {$M_w'$};
    \end{scope}

    \begin{scope}[shift={(5.5,-7.0)}]
        \node (u) at (0,0) {$u$};
        \node (v1) at (1,1) {$v_1$};
        \node (vn1) at (1,-1) {$v_{n-1}$};
        \node (w) at (2,0) {$w$};
        \node (v2) at (2.5,1) {$v_2$};
        \node (v3) at (4,1) {$v_3$};
        \node (vn2) at (2.5,-1) {$v_{n-2}$};
        \node (vn3) at (4,-1) {$v_{n-3}$};
        \node[scale=1.5] at (4,0.2) {$\vdots$};

        \draw[highlight] (v1) -- (v2);
        \draw[highlight] (w) -- (vn1);
        \draw[highlight] (vn2) -- (vn3);

        \draw (u) -- (v1) -- (w) -- (vn1) -- (u);
        \draw (v1) -- (v2) -- (v3);
        \draw (vn1) -- (vn2) -- (vn3);
        
        \node[anchor=north] at (2,-1.5) {$M_u$};
    \end{scope}
    
    \begin{scope}[shift={(5.3, -12.5)}, scale=0.9]
        \def\Rx{4.5} 
        \def\Ry{2.5}

        \node (Mu)       at (162:\Rx cm and \Ry cm) {$M_u$};
        \node (Mw_prime) at (198:\Rx cm and \Ry cm) {$M_w'$};
        \draw[thick] (Mu) -- (Mw_prime);

        \node (Mw)       at (18:\Rx cm and \Ry cm)  {$M_w$};
        \node (Mu_prime) at (-18:\Rx cm and \Ry cm) {$M_u'$};
        \draw[thick] (Mw) -- (Mu_prime);

        \node (M2)       at (125:\Rx cm and \Ry cm) {$M_2$};
        \node (M4)       at (100:\Rx cm and \Ry cm) {$M_4$};
        \node (dots_top) at (75:\Rx cm and \Ry cm)  {$\dots$};
        \node (Mn-2)     at (45:\Rx cm and \Ry cm)  {$M_{n-2}$};

        \node (M2_prime) at (-127:\Rx cm and \Ry cm) {$M_2'$};
        \node (M4_prime)   at (-100:\Rx cm and \Ry cm) {$M_4'$};
        \node (dots_bot)   at (-75:\Rx cm and \Ry cm)  {$\dots$};
        \node (Mn-2_prime)   at (-51:\Rx cm and \Ry cm)  {$M_{n-2}'$};

        \draw[thick] (Mu) -- (M2) -- (M4) -- (dots_top) -- (Mn-2) -- (Mw);
        \draw[thick] (Mu_prime) -- (Mn-2_prime) -- (dots_bot) -- (M4_prime) -- (M2_prime) -- (Mw_prime);
    \end{scope}

\end{tikzpicture}
\caption{Several matchings of $G_n$ and depiction of $\mathcal M_{n/2}(G_n)\cong C_{n+2}$.} \label{fig:matchings-even-cycle}
\end{figure}

Next we use Lemma~\ref{lem:dist-2} to argue that the adjacencies exactly constitute the $(n+2)$-cycle with vertices ordered as $M_u,M_2,\dots,M_{n-2},M_w,M_u',M_{n-2}',\dots,M_2',M_w'$. First, we enumerate the edges that are present. Each $M_i$ is adjacent to $M_{i+2}$ and each $M_i'$ is adjacent to $M_{i+2}'$ because they differ only on the edge matching $v_{i+1}$. We directly observe that 
$G_n\zap{M_u}-v_{n-1}=G_n\zap{M_w'}-v_{n-1}$, $G_n\zap{M_u'}-v_1=G_n\zap{M_w}-v_1$, $G_n\zap{M_u}-v_1=G_n\zap{M_2}-v_1$, $G_n\zap{M_u'}-v_{n-1}=G_n\zap{M_{n-2}'}-v_{n-1}$, $G_n\zap{M_w}-v_{n-1}=G_n\zap{M_{n-2}}-v_{n-1}$, and $G_n\zap{M_w'}-v_1=G_n\zap{M_2'}-v_1$, establishing the additional $6$ edges. 

To check the nonadjacencies, observe that no $M_i$ and $M_j'$ can be adjacent, because even when $v_i$ and $v_j$ have a common neighbor, $M_i$ and $M_j'$ differ on which vertices are matched with $u$ and $w$. Since the unmatched vertex is the same in both $M_u$ and $M_u'$, they cannot be adjacent, and by symmetry neither can $M_w$ and $M_w'$. By direct inspection of both common neighbors of $u$ and $w$, since $G_n\zap{M_u}-v_1\ne G_n\zap{M_w}-v_1$ and $G_n\zap{M_u}-v_{n-1}\ne G_n\zap{M_w}-v_{n-1}$, $M_u$ and $M_w$ are not adjacent. By a symmetric argument, neither are $M_u'$ and $M_w'$. 
We can check that $M_u$ has no neighbors other than $M_2$ and $M_{w'}$ by considering the only other matchings with an isolated vertex that shares a neighbor with $u$, which  are $M_2'$, $M_{n-2}$, $M_{n-2}'$, and $M_w$. We have just checked that $M_w$ is not adjacent, and the others are not either, as $G_n\zap{M_u}-v_1\ne G_n\zap{M_2'}-v_1$, $G_n\zap{M_u}-v_{n-1}\ne G_n\zap{M_{n-2}}-v_{n-1}$, and $G_n\zap{M_u}-v_{n-1}\ne G_n\zap{M_{n-2}'}-v_{n-1}$. Symmetric arguments establish that $M_u'$, $M_w$, and $M_w'$ also have no other neighbors in $\mathcal M_{n/2}(G_n)$, establishing that $\mathcal{M}_{n/2}(G_n)\cong C_{n+2}$ as needed. 
\end{proof}


Having shown that all cycles are realizable as near-perfect matching graphs and hence Bell coloring graphs, we turn to the task of realizing all trees as Bell coloring graphs. To do this, we first describe a procedure for joining near-perfect matching graphs to create larger near-perfect matching graphs. This joining construction relies on a particular property of odd-order graphs. For a graph $G$ on $n=2k+1$ vertices, a vertex $v$ is \emph{uniquely unmatched} if the graph $G\setminus v$ has a unique perfect matching. For two such graphs $H_1$ and $H_2$, Figure~\ref{fig:npm} depicts a near-perfect matching of these graphs joined on their uniquely unmatched vertices. We show that the near-perfect matching graph of this joined graph is exactly the joined near-perfect matching graphs of $H_1$ and $H_2$.

%

\begin{figure}
\centering
\begin{tikzpicture}[scale=0.8,
    thick,
    vertex/.style={circle, fill=black, inner sep=1.5pt},
    matching/.style={line width=2pt}
]
    \node[vertex] (a) at (0, 3.0) {};
    \node[vertex] (b) at (2.0, 3.0) {};
    \node[vertex] (c) at (5.0, 3.0) {};
    \node[vertex] (d) at (7.0, 3.0) {};

    \node[scale=1.5] at (1.0, 2.0) {$\vdots$};
    \node[scale=1.5] at (6.0, 2.0) {$\vdots$};

    \node[vertex] (e) at (0, 0.7) {};
    \node[vertex] (f) at (2.0, 0.7) {};
    \node[vertex] (g) at (5.0, 0.7) {};
    \node[vertex] (h) at (7.0, 0.7) {};

    \node[vertex] (u) at (0,-.2) {};
    \node[vertex] (w1) at (2.0, -0.2) {};
    \node[vertex] (w2) at (3.5, 0.7) {};
    \node[vertex] (i) at (5.0, -0.2) {};
    \node[vertex] (j) at (7.0, -0.2) {};

    \draw[matching] (a) -- (b);
    \draw[matching] (c) -- (d);
    \draw[matching] (e) -- (f);
    \draw[matching] (g) -- (h);
    \draw[matching] (i) -- (j);
    \draw[matching] (w1) -- (w2);
    \node at (3.5, 1.1) {$v_1 \sim v_2$};

    \draw[decorate, decoration={brace, amplitude=10pt, mirror}] (-0.5, -0.7) -- (4.0, -0.7) node[midway, below=10pt, font=\large] {$H_1$};
    \draw[decorate, decoration={brace, amplitude=10pt, mirror}] (3.0, -1.2) -- (7.5, -1.2) node[midway, below=10pt, font=\large] {$H_2$};
\end{tikzpicture}
\caption{Vertices of the graph $G=(H_1\sqcup H_2)/(v_1\sim v_2$) where $H_1$ and $H_2$ have uniquely unmatched vertices $v_1$ and $v_2$, respectively, and a near-perfect matching of $G$ corresponding to a vertex of Type \eqref{item:H1} in $\mathcal L$, the near-perfect matching graphs of $H_1$ and $H_2$ joined on their unique matchings with $v_1$ and $v_2$, respectively, unmatched.} \label{fig:npm}
\end{figure}
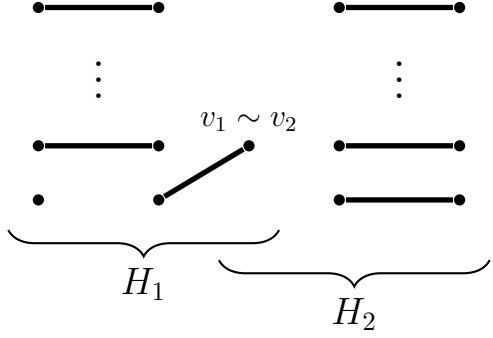

\begin{prop}\label{prop:near-perfect-matchings} 
For $i=1,2$, let $H_i$ be a graph with uniquely unmatched vertex $v_i$, let $k_i=(\abs{V(H_i)}-1)/2$, and let $M_{v_i}\in V(\mathcal M_{k_i}(H_i))$ be the unique near-perfect matching of $H_i$ with $v_i$ unmatched. Then 
\[
\mathcal M_{k_1+k_2}((H_1\sqcup H_2)/(v_1\sim v_2))\cong (\mathcal M_{k_1}(H_1) \sqcup \mathcal M_{k_2}(H_2))/(M_{v_1}\sim M_{v_2}),
\]
where the left-hand side is the near-perfect matching graph of a graph obtained by joining $H_1$ and $H_2$ on their uniquely unmatched vertices, and the right-hand side is a graph obtained by joining the near-perfect matching graphs of $H_1$ and $H_2$ on the matchings corresponding to the uniquely unmatched vertices. 
\end{prop}

\begin{proof}
Let $G\colonequals (H_1\sqcup H_2)/(v_1\sim v_2)$ and let $\mathcal{L} \colonequals (\mathcal{M}_{k_1}(H_1) \sqcup \mathcal{M}_{k_2}(H_2)) /(M_{v_1}\sim M_{v_2})$. The proof proceeds as follows. We first enumerate the types of vertices of $\mathcal L$; then we define a mapping $\Phi$ from $V(\mathcal L)$ to $V(\mathcal M_{k_1+k_2}(G))$ and show that it is bijective. Then we enumerate the edges of $\mathcal L$ and show that $\Phi$ is edge-preserving, establishing an isomorphism between $\mathcal L$ and the near-perfect matching graph of $G$. 

By definition, the vertex set of $\mathcal{L}$ can be partitioned into three types:
\begin{enumerate}
    \item\label{item:H1} Near-perfect matchings of $H_1$ with $v_1$ matched.
    \item\label{item:H2} Near-perfect matchings of $H_2$ with $v_2$ matched.
    \item\label{item:special-vertex} The vertex $M_v$ corresponding to both $M_{v_1}$ and $M_{v_2}$.
\end{enumerate}

Next, we construct a mapping $\Phi$ from $V(\mathcal{L})$ to $V(\mathcal{M}_{k_1+k_2}(G))$. Informally, the distinguished vertex in $\mathcal L$ corresponds to the union of the distinguished matchings in $H_1$ and $H_2$, and every other vertex in $\mathcal L$ corresponds to a near-perfect matching of $G$ by taking the union of a specific near-perfect matching of one base graph with its distinguished vertex matched and the unique matching of the other base graph with its distinguished vertex unmatched. Formally: 

\begin{equation*}
    \Phi(M) \colonequals \begin{cases}
        M\backslash\{v_1w\}\cup\{vw\}\cup M_{v_2} & \text{ if $M$ is of Type~\eqref{item:H1}}\\
        M\backslash\{v_2w\}\cup\{vw\} \cup M_{v_1} & \text{ if $M$ is of Type~\eqref{item:H2}} \\
        M_{v_1}\cup M_{v_2} &\text{ if $M$ is of Type~\eqref{item:special-vertex}}
    \end{cases}
\end{equation*}

Clearly, $\Phi$ is an injection. It remains to describe $\Phi^{-1}$ and establish it as an injection. Let $M$ be an element of $\mathcal{M}_{k_1+k_2}(G)$. Let $\Htil_i$ be the induced subgraph of $G$ on the vertex subset $V(H_i)-v_i+v$ for $i=1, 2$. It must be that $M$ restricted to $\Htil_i$ has $k_i$ edges for $i=1,2$, because each $\Htil_i$ has only enough vertices for $2k_i$ to be matched, and $M$ must have $2(k_1+k_2)$ total vertices matched. 
Since $M$ is a near-perfect matching of $G$, it has exactly one unmatched vertex $w$. We proceed by cases:

\textbf{Case 1.} $w\in \Htil_1$ and $w\notin \Htil_2$. This means $M|_{\Htil_1}$ has $w$ as an unmatched vertex, where $w\ne v$, so $v$ must be matched with a vertex in $\Htil_1$ but not $\Htil_2$. Since $v_2$ is uniquely unmatched in $H_2$, this fixes $M|_{\Htil_2}$ as $M_{v_2}$. Then $M_1\colonequals M|_{\Htil_1} - vw + v_1w$ is the unique vertex in $\mathcal L$ with $\Phi(M_1)=M$.

\textbf{Case 2.} $w\notin \Htil_1$ and $w\in \Htil_2$. This is symmetric to Case 1, with $M_2\colonequals M|_{\Htil_2}-vw+v_2w$ the unique vertex in $\mathcal L$ with $\Phi(M_2)=M$.


\textbf{Case 3.} $w\in \Htil_1$ and $w\in  \Htil_2$. This requires $w=v$. Because  $M|_{\Htil_i}$ has $v_i$ unmatched and $v_i$ is uniquely unmatched for $i=1,2$, we must have $M|_{\Htil_i}=M_{v_i}$ for $i=1,2$, and $M_v$ is the unique vertex in $\mathcal L$ with $\Phi(M)=M$. 

Because there is exactly one vertex in the preimage of $\Phi^{-1}(M)$ for each $M$, $\Phi$ is a bijection.

Finally, we establish $\Phi$ as a graph isomorphism between $\mathcal L$ and the near-perfect matching graph of $G$ by showing that $M_1$ and $M_2$ are adjacent in $\mathcal{L}$ if and only if $\Phi(M_1)$ and $\Phi(M_2)$ are adjacent in  $M_{k_1+k_2}(G)$. 
To do this, we partition the edges of $\mathcal L$ into four types:
\begin{enumerate}
    \item[(i)] Edges between $M_1$ and $M_1'$ both of Type \eqref{item:H1} where $M_1M_1'$ is an edge in $\mathcal{M}_{k_1}(H_1)$.
    \item[(ii)] Edges between $M_2$ and $M_2'$ both of Type \eqref{item:H2} where $M_2M_2'$ is an edge in $\mathcal{M}_{k_2}(H_2)$.
    \item[(iii)] Edges between $M_1$ and $M_v$ where $M_1M_{v_1}$ is an edge in $\mathcal{M}_{k_1}(H_1)$.
    \item[(iv)] Edges between $M_2$ and $M_v$ where $M_2M_{v_2}$ is an edge in $\mathcal{M}_{k_2}(H_2)$.
\end{enumerate}

We reason about each of these types of edges $MM'\in E(\mathcal L)$ separately and show that in each case, $\Phi(M)\Phi(M')$ is an edge in $\mathcal M_{k_1+k_2}(G)$.

\begin{enumerate}
    \item[(i)] If $MM'$ is an edge in $\mathcal M_{k_1}(H_1)$ with $v_1$ matched in both $M$ and $M'$, then there is some $z\in V(H_1)$ such that $H_1\langle M\rangle- z = H_1\langle M'\rangle- z$. Let $\widetilde{z}=z$ if $z\neq v_1$ and $\widetilde{z}=v$ if $z=v_1$. Then $G\langle\Phi(M)\rangle- \widetilde{z} = G\langle\Phi(M')\rangle- \widetilde{z}$, so $\Phi(M)$ and $\Phi(M')$ are adjacent in $\mathcal M_{k_1+k_2}(G)$.
    \item[(ii)] If $MM'$ is an edge in $\mathcal M_{k_2}(H_2)$ with $v_2$ matched in both $M$ and $M'$, then the argument that $\Phi(M)$ and $\Phi(M')$ are adjacent in $\mathcal M_{k_1+k_2}(G)$ is symmetric. 
    \item[(iii)] If $MM'$ is an edge in $\mathcal M_{k_1}(H_1)$ with $v_1$ matched in $M$ and $v_1$ unmatched in $M'$, then $M'=M_v$ and there is some $z\in V(H_1)$ such that $H_1\langle M\rangle- z = H_1\langle M_{v_1}\rangle- z$. Then $G\langle\Phi(M)\rangle- z = G\langle\Phi(M_{v})\rangle- z$, so $\Phi(M)$ and $\Phi(M_{v})$ are adjacent in $\mathcal M_{k_1+k_2}(G)$.
    \item[(iv)]  If $MM'$ is an edge in $\mathcal M_{k_2}(H_2)$ with $v_2$ matched in $M$ and unmatched in $M'$, then the argument that $\Phi(M)$ and $\Phi(M_{v})$ are adjacent in $\mathcal M_{k_1+k_2}(G)$ is symmetric.
\end{enumerate}

Conversely, suppose $MM'$ is an edge in $\mathcal{M}_{k_1+k_2}(G)$. This means there exists a vertex $w$ in $G$ such that $G\langle M_1\rangle- w=G\langle M_2\rangle-w$. We proceed with a case analysis according to the subgraph containing $w$, and the vertices $u$ and $u'$ that are unmatched in $M$ and $M'$, respectively.

\textbf{Case 1.} If $u,u'\in \Htil_1-v$, then $M|_{\Htil_2}=M'|_{\Htil_2}=M_{v_2}$. Therefore, the condition that $G\langle M\rangle- w=G\langle M'\rangle- w$ is equivalent to $\Htil_1\zap{M}- w=\Htil_1\zap{M'}- w$. This is true if and only if the matchings $\Phi^{-1}(M)$ and $\Phi^{-1}(M')$ (which are both in group (i)) are adjacent in $\mathcal{L}$.

\textbf{Case 2.} If $u,u'\in \Htil_2-v$, then the matchings $\Phi^{-1}(M)$ and $\Phi^{-1}(M')$ (which are both in group (ii)) are adjacent in $\mathcal L$ by a symmetric argument.

\textbf{Case 3.} If $u\in \Htil_1-v$ and $u'=v$, then $M|_{\Htil_2}=M'|_{\Htil_2}=M_{v_2}$. Therefore, the condition that $G\langle M\rangle- w=G\langle M'\rangle- w$ is equivalent to $\Htil_1\zap{M}- w=\Htil_1\zap{M'}- w$. This is true if and only if the matchings $\Phi^{-1}(M)$ (which is in group (i)) and $\Phi^{-1}(M')$ (which is in group (iii)) are adjacent in $\mathcal{L}$.

\textbf{Case 4.} If $u\in \Htil_2-v$ and $u'=v$, then the matchings $\Phi^{-1}(M_1)$ (which is in group (ii)) and $\Phi^{-1}(M_2)$ (which is in group (iii)) are adjacent in $\mathcal{L}$ by a symmetric argument.

There is one final case to consider for $MM'$, which we show cannot be an edge in $\mathcal M_{k_1+k_2}(G)$. If $u\in \Htil_1-v$ and $u'\in \Htil_2-v$, then in $M$, $v$ must share a part with a vertex $w_1$ in $\Htil_1$, while in $M'$, $v$ must share a part with a vertex $w_2$ in $\Htil_2$. Therefore $w$ must be $v$. However, in $G\langle M\rangle\setminus \{v\} = G\langle M'\rangle\setminus \{v\}$, we must have that $u$, $u'$, $w_1$, and $w_2$ are all unmatched. Since $v$ can only be adjacent to at most one of these in any given matching, this would yield $3$ unmatched vertices, a contradiction.

We conclude that $\mathcal{M}_{k_1+k_2}(G)$ is isomorphic to $\mathcal{L}$ via the map $\Phi$.
\end{proof}

 We can build many near-perfect matching graphs with the help of Propositions~\ref{prop:near-perfect-matchings}~and~\ref{prop:matchings}. In particular, we can inductively apply Proposition~\ref{prop:near-perfect-matchings} to show that any tree is a near-perfect matching graph of another tree. 

\begin{corollary}\label{cor:tree-subdivision} Let $T$ be a tree of order $n\ge 2$, and let $T_s$ be the tree of order $2n-1$ constructed by subdividing each edge of $T$ once. Then $\mathcal{B}_{n}(\overline{T_s})\cong T$.
\end{corollary}

\begin{proof}
We first argue that the near-perfect matching graph of any subdivided tree is the tree itself. As a base case, observe that $\mathcal M_1(P_3)=P_2$, with the ends of the $P_2$ corresponding to the near-perfect matching with the leaves of $P_3$ unmatched. 

Let $T$ be a tree of order $n$. Because any tree can be described as a leaf attached to a smaller tree, let $T'$ be the tree of order $n-1$ and let $v\in V(T)$ be such that $T=(T'\sqcup P_2)/(v\sim u)$ for either $u\in V(P_2)$. Assume inductively that $T'$ is such that the near-perfect matching graph of $T'_s$ is $T'$ and all vertices in the bipartition of $T'_s$ containing its leaves are uniquely unmatched vertices in $T'_s$. Then applying Proposition~\ref{prop:near-perfect-matchings}, we conclude that the near-perfect matching graph of $T_s$ is $T$. 

Finally, Proposition~\ref{prop:matchings} gives that $T\cong \mathcal M_{n-1}(T_s)\cong \mathcal B_n(\overline{T_s})$.
\end{proof}

Proposition~\ref{prop:near-perfect-matchings} can be applied in other ways to yield additional families of graphs realizable as Bell coloring graphs. For example, note that although the graphs yielding even cycles do not have uniquely unmatched vertices, odd cycles do. Therefore, any graph that can be constructed by iteratively joining odd cycles is a Bell coloring graph. We leave further exploration of families of graphs that can be constructed similarly for future work and instead summarize our results of this section as follows.

\begin{theorem}\label{thm:trees-cycles}
    All trees and all cycle graphs are realizable as Bell coloring graphs.
\end{theorem}

\section{Reconstructing trees from their Bell 3-coloring graph}\label{sec:tree-reconstruction}

This section proves Theorem~\ref{thm:tree-reconstruction} through a sequence of lemmas that yield an algorithmic procedure for reconstructing a tree given its Bell 3-coloring graph. First, we prove that three categories of trees can be differentiated based on the degree sequence of their Bell coloring graphs, and then we show that within each category, other Bell coloring graph features determine the specific tree. As we can manually verify the result for all trees of order at most $5$, we assume $n\geq 6$ throughout this section. 

Since $\B_3(T)$ has $2^{n-2}$ vertices (see \cite{DP09}*{Proposition~3.1}), it suffices to show that the map $T\mapsto \B_3(T)$ is injective when restricted to trees of fixed order. Let $Z=\{X, Y, \emptyset\}$ denote the unique bipartition of $T$ viewed as a stable $3$-partition, with $|X|\leq |Y|$. A \emph{double broom} $B(3, a, b)$ is a tree formed by a path $u$-$v$-$w$ with $a$ leaves adjacent to $u$ and $b$ leaves adjacent to $w$. For $n \ge 6$, we have $a+b \ge 3$, and we assume $a \ge b$. Note that $|X|=2$ if and only if $T \cong B(3, a, b)$.

We categorize trees of order $n$ based on the size of the smaller bipartition set $X$:
\begin{enumerate}[(i)]
    \item The star graph $K_{1, n-1}$, where $|X|=1$.
    \item Double brooms $B(3, a, b)$, where $|X|=2$.
    \item Generic trees, where $|X|\geq 3$.
\end{enumerate}

\begin{lemma}\label{lem:max-degree-B3T}
For a tree $T$ of order $n$, the maximum degree of $\B_3(T)$ is at most $n$.
\end{lemma}

\begin{proof}
Let $P\in \B_3(T)$ be given by $P=\{S_1, S_2, S_3\}$. Let $m_{v} = \# \{R \ | \ P\sim_{v} R\}$ count the number of neighbors that differ from $P$ at the vertex $v$. 

We claim that $m_{v}\leq 1$ for each $v\in V(T)$. Let $w$ be a neighbor of $v$. Without loss of generality, suppose $v\in S_1$ and $w\in S_2$. If there exists a stable partition $R$ (different from $P$) for which $P-v=R-v$, then it is obtained from $P$ by moving $v$ to $S_3$; in this case, $R=\{S_1-v, S_2, S_3\cup \{v\}\}$ and $m_v=1$. If $v$ has neighbors in both $S_2$ and $S_3$, then no such $R$ exists and $m_v=0$. In either case, $m_v\leq 1$ for each $v$, and so $\deg(P)\leq \sum_{v} m_v \leq n$.
\end{proof}

We now establish that the unique bipartition $Z=\{X, Y, \emptyset\}$ is the sole candidate for a vertex of degree $n$.

\begin{lemma}\label{lem:degree-n-implies-bipartition}
For a tree $T$ of order $n$, let $P\in \B_3(T)$ be a partition with $3$ nonempty parts $S_1, S_2, S_3$. Then $\deg_{\B_3(T)}(P)\leq n-1$.
\end{lemma}

\begin{proof}
Suppose, to the contrary, that some partition $P=\{S_1,S_2,S_3\}$ with $S_1,S_2,S_3$ nonempty has degree $n$ as a vertex in $\B_3(T)$. By the proof of Lemma~\ref{lem:max-degree-B3T}, any partition $P\in \B_3(T)$ with $\deg(P)=n$ satisfies the following property: the neighborhood $N_{T}(v)$ is contained within a single part of $P$ for each vertex $v$ of $T$. 

The restriction of $P$ to any subtree of $T$ maintains this property. Consider a minimal (with respect to inclusion) induced subtree $T'$ of $T$ such that the restriction of $P$ to $T'$ also has 3 nonempty parts. In particular, $|V(T')|\geq 3$. Consider a leaf vertex $u$ of $T'$, with unique neighbor $v$. Suppose without loss of generality that $u\in S_1$ and $v\in S_2$. Now consider any neighbor $w$ of $v$ (excluding $u$). By the desired property, $w\in S_1$, but then consider $T''=T'-u$. Since both $u$ and $w$ belong to $S_1$, we know that $T''$ also has three nonempty parts and has the desired property, contradicting the minimality of $T'$. 
\end{proof}

The degree of $Z=\{X, Y, \emptyset\}$ in $\B_3(T)$ depends on the size of the smaller bipartition set.  In particular, we use the following notation for a neighbor of $Z$: 
\[
Z_v\colonequals 
\begin{cases}
\{X-v,\,Y,\,\{v\}\}, & v\in X,\\[2pt]
\{X,\,Y-v,\,\{v\}\}, & v\in Y,
\end{cases}
\] 
Then we can write the degree as $\deg(Z)\colonequals\abs{\set{Z_v \, : \, v\in V(T)}\setminus \{Z\}}$.

\begin{lemma} \label{lem:degree-of-bipartition}
Let $Z$ be the bipartition of a tree $T$. The degree of $Z$ in $\B_3(T)$ is given by:
\[
\deg(Z) = 
\begin{cases} 
n & \text{if } |X| \ge 3 \text{ (generic)} \\
n-1 & \text{if } |X| = 2 \text{ (double broom)}\\
n-1 & \text{if } |X| = 1 \text{ (star)}
\end{cases}
\]
\end{lemma}

\begin{proof}
The neighbors of $Z$ are exactly the partitions obtained by moving a vertex $v$ to the empty part. We analyze when these partitions are distinct. First, $Z_v=Z$ if and only if $|X|=1$ (and $v\in X$). Note that $|Y|=1$ is impossible as $n\ge 6$. For $u\neq w$, $Z_u=Z_w$ if and only if $u, w$ are in the same part, say $X$, and $X=\{u, w\}$ (i.e., $|X|=2$).

\begin{enumerate}[(i)]
     \item $|X|=1$: The move from $X$ yields $Z$ itself. Moves from $Y$ yield $|Y|=n-1$ distinct neighbors. Thus, $\deg(Z)=n-1$. 
    \item $|X|=2$: Moves from $X$ yield a single distinct neighbor ($Z_u=Z_w$). Moves from $Y$ yield $|Y|=n-2$ distinct neighbors. Thus, $\deg(Z)=1+(n-2)=n-1$. 
    \item $|X|\geq 3$: All $n$ moves yield distinct neighbors, none equal to $Z$. Thus, $\deg(Z)=n$.
    \qedhere
\end{enumerate}
\end{proof}

These bounds allow us to classify trees based on the degree sequence of $\B_3(T)$ as follows.

\begin{lemma}\label{lem:degree-classification}
Consider the degree sequence $d_1\geq d_2\geq d_3\geq \dots \geq d_{2^{n-2}}$ of some $\B_3(T)$. Then:
\begin{enumerate}[(i)]
\item $T$ is a star if and only if $d_1=d_{2^{n-2}}$.
\item $T$ is a double broom $B(3, a, b)$ if and only if $d_1=d_2>d_{2^{n-2}}$.
\item $T$ is generic if and only if $d_1>d_2$.
\end{enumerate}
\end{lemma}

\begin{proof} We analyze the three cases based on the size of the smaller bipartition set $X$.

\textbf{Case 1:} $|X|=1$.

Let $X=\{c\}$ (the center). By Lemma \ref{lem:degree-of-bipartition}, $\deg(Z)=n-1$. Any stable 3-partition $P$ must isolate $c$, say $P=\{\{c\}, L_1, L_2\}$, where $L_1, L_2$ partition $Y$ (the leaves). The center $c$ cannot move to create a distinct partition. Any leaf $v\in L_i$ can move to $L_j$ ($j\neq i$) because $N(v)=\{c\}$. Therefore, $\deg(P) = |L_1|+|L_2| = |Y| = n-1$. The graph $\B_3(T)$ is $(n-1)$-regular. Thus, $d_1=d_{2^{n-2}}$.  

\textbf{Case 2:} $|X|=2$

By Lemma \ref{lem:degree-of-bipartition}, $\deg(Z)=n-1$. By Lemmas~\ref{lem:max-degree-B3T} and \ref{lem:degree-n-implies-bipartition}, the maximum degree must be $n-1$, so $d_1=n-1$. We first show $d_1=d_2$. Let $X=\{u, w\}$. Let $x$ be a leaf adjacent to $u$. Consider the partition $Z_x = \{X, Y-x, \{x\}\}$. We compute $\deg(Z_x)$:
\begin{itemize}
    \item $x$ moves to $Y-x$ and yields $Z$. (1 neighbor).
    \item $w$ moves to $\{x\}$, since $wx\notin E(T)$. (1 neighbor).
    \item Any $y\in Y-x$ moves to $\{x\}$. ($n-3$ neighbors).
\end{itemize}
Total degree is $1 + 1 + (n-3) = n-1$. Since $Z$ and $Z_x$ are distinct vertices with degree $n-1$, we have $d_1=d_2=n-1$.

We now show that the Bell 3-coloring graph is not regular. Consider $Z_v = \{X, Y-v, \{v\}\}$ where $v$ is the central vertex in $Y$. Since $v$ is adjacent to both $u$ and $w$, neither can move to $\{v\}$. The vertex $v$ can move from $\{v\}$ back into $Y-v$, which yields $Z$, and each leaf in $Y-v$ can move to $\{v\}$. Hence, $\deg(Z_v) = 1 + (n-3) = n-2$. Thus, $\B_3(T)$ is not regular, and $d_1=d_2>d_{2^{n-2}}$.

\textbf{Case 3:} $|X|\ge 3$  

By Lemma \ref{lem:degree-of-bipartition}, $\deg(Z)=n$, while any other partition $P$ has three nonempty parts and therefore satisfies $\deg(P) \leq n-1$ by Lemma~\ref{lem:degree-n-implies-bipartition}. Thus, $Z$ is the unique vertex of maximum degree. Therefore, $d_1=n > d_2$. \end{proof}

We now reconstruct the adjacency matrix of $T$. By Lemma~\ref{lem:degree-classification}, we can distinguish the three categories of trees; we treat each case separately. Before we proceed, we explicitly describe the Bell $2$-coloring graph of an independent set.

\begin{ex}\label{ex:Bell-2-color-indep} Given an integer $m\geq 1$, the graph $\B_2(\overline{K}_{m})$ is obtained from the hypercube $Q_{m-1}$ by adding edges between antipodal vertices. In particular, $\B_2(\overline{K}_{m})$ is $m$-regular.

Write $V(\overline{K}_{m})=\{v_1,\dots,v_{m}\}$ and fix $v_{m}$ to break
symmetry. Any stable $2$-partition of $\overline{K}_{m}$ can be written uniquely as
\[
\{S_1,S_2\cup\{v_{m}\}\},
\]
where $S_1,S_2\subseteq\{v_1,\dots,v_{m-1}\}$ form a partition. Encode such a partition by a binary string $s_1\cdots s_{m-1}\in\{0,1\}^{m-1}$, where
\[
s_i= \begin{cases}
0, & v_i\in S_1,\\
1, & v_i\in S_2.
\end{cases}
\]
This yields a bijection between $V\bigl(\B_2(\overline{K}_{m})\bigr)$ and the
vertex set of the $(m-1)$-dimensional hypercube $Q_{m-1}$, so
\[
|V(\B_2(\overline{K}_{m}))| = 2^{m-1}.
\]
We now translate adjacency. Let $P$ and $P'$ be two such partitions, with corresponding strings $s_1\cdots s_{m-1}$ and $t_1\cdots t_{m-1}$.

\emph{(i) Moving $v_i$ for $i\le m-1$.}
If $P'$ is obtained from $P$ by moving $v_i$ from $S_1$ to $S_2\cup\{v_{m}\}$
(or vice versa), then $s_j=t_j$ for all $j\ne i$ and $s_i\ne t_i$.
Thus, $P$ and $P'$ differ at exactly one leaf $v_i$, and the strings
differ in exactly one coordinate. These edges correspond to the usual edges of $Q_{m-1}$.

\emph{(ii) Moving $v_{m}$.}
The only way $v_{m}$ can move is by swapping the two parts:
\[
\{S_1,S_2\cup\{v_{m}\}\} \longleftrightarrow
\{S_2\cup\{v_{m}\},S_1\}.
\]
Under our encoding, this swap sends $s_1\cdots s_{m-1}$ to its bitwise complement $\overline{s}_1\cdots \overline{s}_{m-1}$, i.e., a string that differs from $s$ in all coordinates. Thus, each vertex in $\B_2(\overline{K}_{m})$ has a unique additional neighbor given by its antipode in $Q_{m-1}$.

Consequently, $\B_2(\overline{K}_{m})$ is obtained from $Q_{m-1}$ by adding an edge between every pair of antipodal vertices. Every vertex has degree $(m-1)$ from the hypercube edges plus $1$ from its antipodal neighbor, so $\B_2(\overline{K}_{m})$ is $m$-regular.
\end{ex}

\subsection{Stars}

By Lemma~\ref{lem:degree-classification}, stars are the only trees whose Bell 3-coloring graphs are regular. Thus, we can immediately infer when a Bell 3-coloring graph of a tree of order $n$ is the Bell 3-coloring graph of $K_{1,n-1}$, the unique star of order $n$. 

Before handling the reconstruction of double brooms from their Bell 3-coloring graphs, we observe the exact structure of a star's Bell 3-coloring graph. Note that any stable 3-partition of a star must isolate the center vertex, so stable 3-partitions of a star are in bijection with stable 2-partitions of $n-1$ isolated vertices and hence $\B_3(K_{1, n-1})\cong \B_2(\overline{K}_{n-1})$, which Example~\ref{ex:Bell-2-color-indep} establishes as $Q_{n-2}$ with edges between antipodal vertices.

\subsection{Double brooms} We now handle the case of double brooms (see Figure~\ref{fig:double_broom}). Because a double broom $B(3,a,b)$ of order $n$ is determined by $a$, it suffices to show that the number of edges in the Bell 3-coloring graph of $B(3,a,b)$ determines $a$. We prove the following lemma and note that from this edge count, we can uniquely recover $a$.

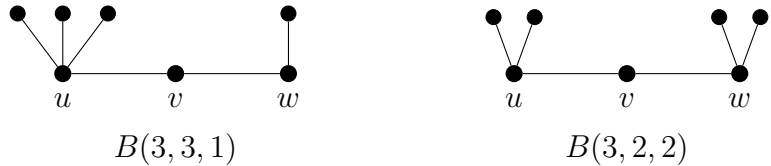
\begin{figure}[h]
    \centering
    \begin{tikzpicture}[
        main_vertex/.style={circle, draw, fill=black, thick, inner sep=2pt},
        leaf_vertex/.style={circle, draw,  fill=black,  inner sep=2pt},
        scale=1.0
    ]

    \begin{scope}
        \node[main_vertex, label=below:$u$] (u1) at (-1.5, 0) {};
        \node[main_vertex, label=below:$v$] (v1) at (0, 0) {};
        \node[main_vertex, label=below:$w$] (w1) at (1.5, 0) {};
        \draw (u1)--(v1)--(w1);

        \node[leaf_vertex] (l1a) at (-2.1, 0.8) {};
        \node[leaf_vertex] (l1b) at (-1.5, 0.8) {};
        \node[leaf_vertex] (l1c) at (-0.9, 0.8) {};
        \draw (u1)--(l1a); \draw (u1)--(l1b); \draw (u1)--(l1c);

        \node[leaf_vertex] (l1d) at (1.5, 0.8) {};
        \draw (w1)--(l1d);
         \node at (0, -1) {$B(3, 3, 1)$};
    \end{scope}

    \begin{scope}[xshift=6cm]
         \node[main_vertex, label=below:$u$] (u2) at (-1.5, 0) {};
        \node[main_vertex, label=below:$v$] (v2) at (0, 0) {};
        \node[main_vertex, label=below:$w$] (w2) at (1.5, 0) {};
        \draw (u2)--(v2)--(w2);

        \foreach \angle in {70, 110} {
            \node[leaf_vertex] (ua\angle) at ($(u2) + (\angle:0.8)$) {};
            \draw (u2) -- (ua\angle);
            \node[leaf_vertex] (wa\angle) at ($(w2) + (180-\angle:0.8)$) {};
            \draw (w2) -- (wa\angle);
        }
         \node at (0, -1) {$B(3, 2, 2)$};
    \end{scope}

    \end{tikzpicture}
    \caption{Examples of double broom graphs $B(3, a, b)$, constructed from a central path $u-v-w$ by adding $a$ leaves to $u$ and $b$ leaves to $w$.}
    \label{fig:double_broom}
\end{figure}

\begin{lemma}\label{lem:double-broom-reconstruction}
The number of edges in the Bell 3-coloring graph of $T=B(3, a, b)$ is
\[
    |E(\B_3(T))| = (2(a+b)+1)2^{a+b-1} + 2^a + 2^b - 1.
\]
\end{lemma}

\begin{proof} 
The vertex set $V(\B_3(T))$ partitions into two sets based on the placement of the nonadjacent centers $u$ and $w$:
\begin{align*}
    V_0 &= \{P \in \B_3(T) : u, w \text{ lie in the same part of } P\}, \\
    V_1 &= \{P \in \B_3(T) : u, w \text{ lie in different parts of } P\}.
\end{align*}
Since $v$ is adjacent to both $u$ and $w$, any $P \in V_0$ must contain the part $\{u,w\}$; the remaining $n-2$ vertices form an independent set $S$. Thus the induced subgraph $\B_3(T)[V_0]$ is isomorphic to $\B_2(S) \cong \B_2(\overline{K}_{a+b+1})$. From Example~\ref{ex:Bell-2-color-indep}:
\[
    |E(V_0)| = (a+b+1)2^{a+b-1}.
\]
For $P \in V_1$, the vertices $u, v, w$ must occupy distinct parts. The $a$ leaves at $u$ have 2 available colors each, and the $b$ leaves at $w$ have 2 available colors each. No adjacencies exist between these leaves. Thus, the induced subgraph $\B_3(T)[V_1]$ is isomorphic to the hypercube $Q_{a+b}$, so:
\[
    |E(V_1)| = (a+b)2^{a+b-1}.
\]
Finally, edges between $V_0$ and $V_1$ occur when $u$ moves to the part containing $w$ (possible only if no neighbor of $u$ is in that part) or vice versa.
A partition $P \in V_1$ admits a move of $u$ to $w$'s part if and only if all $a$ neighbors of $u$ are in $v$'s part. There are $2^b$ such partitions.
Similarly, $P$ admits a move of $w$ to $u$'s part if and only if all $b$ neighbors of $w$ are in $v$'s part. There are $2^a$ such partitions.
Exactly one partition (where all leaves are with $v$) allows both moves, producing the same neighbor in $V_0$. Thus, the number of cross-edges is:
\[
    |E(V_0, V_1)| = 2^a + 2^b - 1.
\]
Summing these gives the total edge count:
\[
    |E(\B_3(T))| = (2(a+b)+1)2^{a+b-1} + 2^a + 2^b - 1. \qedhere 
\]
\end{proof}

Using Lemma~\ref{lem:double-broom-reconstruction}, we can distinguish two double brooms $B(3, a, b)$ and $B(3, c, d)$. Indeed, suppose $B(3, a, b)$ and $B(3,c, d)$ have isomorphic Bell $3$-coloring graphs. Since the number of vertices $2^{n-2}$ fixes the sum $a+b$, the term $(2(a+b)+1)2^{a+b-1}$ is invariant. The equality of edge counts implies $2^a + 2^b = 2^c + 2^d$, which leads to the equality $\{a,b\}=\{c, d\}$, as desired.

\subsection{Generic trees}  Suppose $|Y|\geq |X|\geq 3$. From the proof of Lemma~\ref{lem:degree-of-bipartition}, the neighbors of the bipartition $Z$ in $\B_3(T)$ are exactly the partitions
\[
Z_w \colonequals
\begin{cases}
\{X-w,\,Y,\,\{w\}\} & \text{ if }w\in X,\\[2pt]
\{X,\,Y-w,\,\{w\}\} & \text{ if }w\in Y,
\end{cases}
\]
each obtained by moving a single vertex $w$ to the third part.

Lemma~\ref{lem:generic-tree-reconstruction} shows that a square in $\B_3(T)$ incident to $Z$ as in Figure~\ref{fig:ZuZvR-diamond} certifies the nonadjacency of a pair of vertices in $T$.

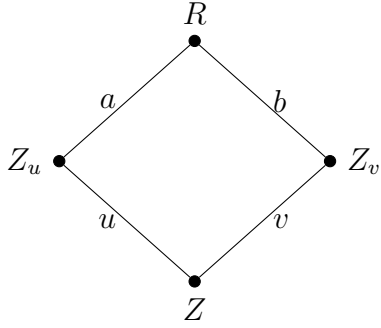
\begin{figure}[ht]
\centering
\begin{tikzpicture}[
    vertex/.style={circle, draw, fill=black, inner sep=1.5pt},
    scale=1.0
]
  \node[vertex, label=below:$Z$]   (Z)  at (0,0) {};
  \node[vertex, label=left:$Z_u$]  (Zu) at (-1.8,1.6) {};
  \node[vertex, label=right:$Z_v$] (Zv) at (1.8,1.6) {};
  \node[vertex, label=above:$R$]   (R)  at (0,3.2) {};

  \draw (Z)  -- node[left]  {$u$} (Zu);
  \draw (Z)  -- node[right] {$v$} (Zv);
  \draw (Zu) -- node[left]  {$a$} (R);
  \draw (Zv) -- node[right] {$b$} (R);
\end{tikzpicture}
\caption{Square formed by $Z$, $Z_u$, $Z_v$, and a common neighbor $R$ with responsible vertices labeled.}
\label{fig:ZuZvR-diamond}
\end{figure}

\begin{lemma}\label{lem:generic-tree-reconstruction}
Let $T$ be a tree of order $n\ge 6$ with bipartition $Z = \{X,Y,\emptyset\}$ where $|X|\ge 3$ and $|Y|\ge 3$. For distinct vertices $u,v\in V(T)$, the vertices $Z_u$ and $Z_v$ have a common neighbor in $\B_3(T)$ distinct from $Z$ if and only if $u$ and $v$ are nonadjacent in $T$.
\end{lemma}

\begin{proof} \emph{($\Leftarrow$)} Suppose $u$ and $v$ are nonadjacent. Define the stable $3$-partition:
\[
R \colonequals \{X\setminus\{u,v\},\, Y\setminus\{u,v\},\, \{u,v\}\}.
\]
Because $|X|,|Y|\ge 3$, both $X\setminus\{u,v\}$ and $Y\setminus\{u,v\}$ are nonempty, so $R\neq Z$. Moreover,
\[
R-v = \{X\setminus\{u,v\},\,Y\setminus\{u,v\},\,\{u\}\} = Z_u - v,
\]
so $Z_u\sim_v R$, and symmetrically $Z_v\sim_u R$. Thus, $R\neq Z$ is a common neighbor of $Z_u$ and $Z_v$.

\emph{($\Rightarrow$)} Conversely, suppose $Z_u$ and $Z_v$ have a common neighbor $R\neq Z$:
\[
Z_u \sim_a R,\qquad Z_v \sim_b R
\]
for some vertices $a,b$.

As $T$ is generic, no edge of $\B_3(T)$ is doubly realized: Lemma~\ref{lem:double_edge} requires a stable partition with parts $\{x\},\{y\},V(T)\setminus\{x,y\}$, which is impossible because $|X|,|Y|\geq 3$. Consequently, each edge has a \emph{unique} responsible vertex, so $a$ and $b$ are uniquely determined. From
\[
Z_u - a = R - a \quad\text{and}\quad Z_v - b = R - b
\]
we deduce
\[
Z_u - \{a,b\} = R - \{a,b\} = Z_v - \{a,b\}.
\]
But $Z_u$ and $Z_v$ differ from $Z$ only at $u$ and $v$, so they can only differ from each other at $u$ and $v$. Thus, $\{a, b\}=\{u, v\}$. If $a=u$, then $Z_u - u = Z - u = R - u$, so $R$ is obtained from $Z-u$ by deciding whether $u$ is in its own part or not; the only stable options are $Z$ and $Z_u$, contradicting the assumption that $R\notin \{Z, Z_u\}$. We deduce that $a\neq u$, and similarly $b\neq v$. So, $a=v$ and $b=u$, giving us $Z_u \sim_v R$ and $Z_v \sim_u R$.

Now, assume for a contradiction that $u$ and $v$ are adjacent. Consider the possible moves:

\begin{itemize}
    \item In $Z_u = \{X-u,Y,\{u\}\}$, the vertex $v$ lies in $Y$ and all neighbors of $v$ lie in $X$, including $u$ in the singleton part $\{u\}$. Any legal move of $v$ must send $v$ to a part containing no neighbor of $v$.  
  If $v$ has a neighbor in $X\setminus\{u\}$, then $v$ has no legal move at all. If $v$ is a leaf with a unique neighbor $u$, there is exactly one legal move, sending $v$ to $X\setminus\{u\}$, which yields a unique neighbor $R_1$ of $Z_u$ via $v$.
  \item  Symmetrically, in $Z_v = \{X,Y-v,\{v\}\}$, any legal move of $u$ (if it exists) produces a unique neighbor $R_2$ of $Z_v$ via $u$, and this can occur only if $u$ is a leaf with a unique neighbor $v$.
\end{itemize}
In a tree with at least three vertices, two adjacent vertices cannot both be leaves, so at most one of $R_1$ and $R_2$ exists, and in any case $R_1\neq R_2$. Thus, there is no partition $R\neq Z$ that is adjacent to both $Z_u$ via $v$ and $Z_v$ via $u$, a contradiction.
\end{proof}

This lemma is particularly powerful when combined with our ability to identify from $\B_3(T)$ the unique vertex $Z$ whose corresponding partition has an empty part. We can identify this vertex because Lemma~\ref{lem:degree-of-bipartition} says that it has degree $n$, and Lemma~\ref{lem:degree-n-implies-bipartition} ensures it is the unique vertex of maximum degree. With $Z$ identified, we can label each of its neighbors in $\B_3(T)$ by a vertex of $T$ and then apply Lemma~\ref{lem:generic-tree-reconstruction} to identify exactly those pairs of vertices that are adjacent. This adjacency matrix completely describes the tree, and the corollary below follows. 

\begin{corollary}
    Let $T_1$ and $T_2$ be two trees of order $n$ with bipartitions $Z_1=\{X_1,Y_1,\emptyset\}$ and $Z_2=\{X_2,Y_2,\emptyset\}$ with $\abs{X_1},\abs{X_2},\abs{Y_1},\abs{Y_2}\ge 3$. Then $\B_3(T_1)\cong \B_3(T_2)$ if and only if $T_1\cong T_2$.
\end{corollary}

We now combine the classification of tree types with the specific reconstruction results for each case to prove the main result on Bell $3$-colorings of trees.

\begin{theorem}\label{thm:tree-reconstruction}
For trees $T_1$ and $T_2$, we have $\B_3(T_1)\cong \B_3(T_2)$ if and only if $T_1\cong T_2$.
\end{theorem}

\begin{proof}
One of the implications is immediate. For the converse, assume $\B_3(T_1) \cong \B_3(T_2)$. Since the graphs have $2^{n-2}$ vertices, $T_1$ and $T_2$ must have the same order. By Lemma~\ref{lem:degree-classification}, $T_1$ and $T_2$ must belong to the same class of trees (generic, double broom, or star) based on the degree sequence of the Bell coloring graph.

If they are stars, they are trivially isomorphic as they share the same order. If they are double brooms, Lemma~\ref{lem:double-broom-reconstruction} shows that nonisomorphic double brooms yield Bell coloring graphs with different edge counts, forcing $T_1 \cong T_2$. If they are generic, Lemma~\ref{lem:generic-tree-reconstruction} allows the unique reconstruction of the adjacency matrix of the tree from the common neighbors of the vertices adjacent to the unique bipartition vertex $Z$.
\end{proof}

\section{Reconstruction from the Bell Multigraph}\label{sec:reconstruction-multigraph}

In this section, we prove that a graph $G$ of order $n$ can be reconstructed (up to universal vertices) from its Bell coloring multigraph $\BellMulti{n}{G}$. We first identify the singleton partition $Q=\{\{v_1\},\dots,\{v_n\}\}$ (up to automorphism), use the neighborhood of $Q$ in $\BellMulti{n}{G}$ to determine the line graph $L(\overline{G})$, and finally invoke Whitney's isomorphism theorem~\cite{Whi32} to recover $\overline{G}$ from its line graph. This last step requires distinguishing triangles in the line graph that correspond to $K_3$ versus $K_{1,3}$, as these are the only nonisomorphic connected graphs with isomorphic line graphs~\cite{Whi32}.

Throughout we fix a graph $G$ with vertex set $\{v_1,\dots,v_n\}$ and work with $\BellMulti{n}{G}$ and its underlying simple graph $\B_n(G)$.
Let
\[
Q \colonequals \bigl\{\{v_1\}, \{v_2\}, \dots,\{v_n\}\bigr\}
\]
be the \emph{singleton partition}. We repeatedly use the double-edge characterization (Lemma~\ref{lem:double_edge}): an edge $PP'$ has multiplicity $2$ if and only if, up to relabeling parts, one of $P,P'$ contains $\{a,b\}$ and $\emptyset$ while the other contains $\{a\}$ and $\{b\}$ for some nonedge $ab$ of $G$, and $P,P'$ agree on all other parts.

\subsection{Totally-doubled partitions}

Before we are ready to construct $L(\overline{G})$ from $\BellMulti{n}{G}$, we identify properties of the singleton partition and structurally similar partitions called totally-doubled partitions.

\begin{defn}
A vertex $P\in V(\BellMulti{n}{G})$ is \emph{totally-doubled} if every edge incident with $P$ has multiplicity $2$. Let $\mathcal{D}(G)$ denote the set of totally-doubled partitions.
\end{defn}

The partition $Q$ lies in $\mathcal{D}(G)$: every neighbor of $Q$ is obtained by merging two singletons $\{u\},\{v\}$ with $u\nsim v$, and by Lemma~\ref{lem:double_edge} such an edge has multiplicity $2$.

\begin{lemma}\label{lem:D-max-two}
Let $P\in\mathcal{D}(G)$. Then every nonempty part of $P$ has size at most $2$.
\end{lemma}

\begin{proof}
Suppose $P$ has a part $X$ with $|X|\ge3$, and choose $w\in X$. Since $P$ has $n$ parts and $|X|\ge 3$, at least one part is empty. Move $w$ from $X$ into an empty part to obtain $P'$. Then $P'$ is stable and $P\sim_w P'$.

In this move, only $w$ changes its part. By Lemma~\ref{lem:double_edge}, however, an edge of multiplicity $2$ requires toggling a pair $\{a,b\}$ into singletons $\{a\}$ and $\{b\}$, so two vertices change parts. Thus $PP'$ has multiplicity $1$, contradicting $P\in\mathcal{D}(G)$.
\end{proof}

We now characterize the partitions in $\mathcal{D}(G)$ that maximize the degree in $\B_n(G)$. Given two distinct vertices $u, w$ in a graph $G$, we say that $u$ and $w$ are \emph{dominating twins} if 
\[
N(u) = N(w) = V(G)\setminus \{u, w\}.
\]
By definition, $u$ and $w$ must be nonadjacent. 

The following lemma shows that the totally-doubled vertices of maximum degree consist only of singleton parts and pairs, where each pair is a dominating twin.

\begin{lemma}\label{lem:pair-dominating-twins-max}
Let $P\in\mathcal{D}(G)$ be a vertex that maximizes the degree in $\B_n(G)$ among all vertices in $\mathcal{D}(G)$.
Write
\[
P =\bigl\{\{v_1\},\dots,\{v_r\},\{u_1,w_1\},\dots,\{u_s,w_s\}\bigr\},
\]
where all nonempty parts have size $1$ or $2$. Then for each $i$, the vertices $u_i$ and $w_i$ are dominating twins.
\end{lemma}

\begin{proof}
Since $P$ is stable, $u_iw_i\notin E(G)$ for each $i$.

\smallskip\noindent
\emph{Step 1 (dominance).}
Fix $i$ and suppose there exists $x\in V(G)\setminus\{u_i,w_i\}$ with $xu_i\notin E(G)$ and $xw_i\notin E(G)$. Then $\{u_i,w_i,x\}$ is stable. Let $P'$ be obtained from $P$ by moving $x$ into the part $\{u_i,w_i\}$. Then $P'$ is stable and $P\sim_x P'$.

Only $x$ changes its part between $P$ and $P'$, so the pattern from Lemma~\ref{lem:double_edge} does not occur. Thus, $PP'$ has multiplicity~$1$, contradicting $P\in\mathcal{D}(G)$. Hence
\[
N(u_i)\cup N(w_i)=V(G)\setminus\{u_i,w_i\}.
\]
\smallskip\noindent
\emph{Step 2 (twins).}
Suppose there exists
$x\in V(G)\setminus\{u_i,w_i\}$ with $xu_i\in E(G)$ and $xw_i\notin E(G)$. 

We first describe the neighbors of $P$ in $\B_n(G)$. Since $P \in \mathcal{D}(G)$, every incident edge has multiplicity $2$. By Lemma~\ref{lem:double_edge}, any neighbor is obtained via one of the following operations:
\begin{enumerate}
  \item[(a)] \emph{splitting a $2$-element part} $\{a,b\}$:
  replace $\{a,b\}$ and an empty part by the singletons $\{a\}$ and $\{b\}$;
  \item[(b)] \emph{merging two singletons} $\{a\},\{b\}$ with $ab\notin E(G)$:
  replace them by $\{a,b\}$ and an empty part.
\end{enumerate}
Define $S(P)\subseteq E(\overline{G})$ as the set of nonedges $ab$ of $G$ such that either $P$ contains the part $\{a,b\}$, or $P$ contains $\{a\}$ and $\{b\}$ as singleton parts. Then
\[
\deg_{\B_n(G)}(P)=|S(P)|.
\]
For $Q$, every nonedge $ab$ of $G$ gives a neighbor $Q_{ab}$ obtained by merging $\{a\}$ and $\{b\}$, so
\[
\deg_{\B_n(G)}(Q)=|E(\overline{G})|.
\]
In $P$, the vertex $w_i$ lies in the part $\{u_i,w_i\}$, so $w_i$ is not a singleton, and $\{x,w_i\}$ is not a part of $P$. Thus, $xw_i\notin S(P)$, while $xw_i$ is a nonedge of $G$. Hence, $|E(\overline{G})|>|S(P)|$, and
\[
\deg_{\B_n(G)}(P)
=|S(P)|
<|E(\overline{G})|
=\deg_{\B_n(G)}(Q),
\]
contradicting the maximality of $P$ in $\mathcal{D}(G)$. Therefore, no vertex is adjacent to exactly one of $u_i$ and $w_i$, and together with dominance we obtain
\[ 
N(u_i)=N(w_i)=V(G)\setminus\{u_i,w_i\}.  \qedhere
\]
\end{proof}

We are now ready to show that maximum degree totally-doubled vertices are indistinguishable from the singleton partition in a Bell coloring multigraph in the following sense.

\begin{lemma}\label{lem:involution-max-degree}
Let $P\in\mathcal{D}(G)$ have maximum degree in $\B_n(G)$ among all vertices in $\mathcal{D}(G)$. Then there exists an automorphism $\Psi$ of $\BellMulti{n}{G}$ such that $\Psi(P)=Q$.
\end{lemma}

\begin{proof}
By Lemmas~\ref{lem:D-max-two} and \ref{lem:pair-dominating-twins-max} we may write
\[
P=\bigl\{\{v_1\},\dots,\{v_r\},\{u_1,w_1\},\dots,\{u_s,w_s\}, \text{ empty parts}\bigr\},
\]
where each pair $\{u_i,w_i\}$ consists of (nonadjacent) dominating twins.

Let $R$ be an arbitrary stable partition of $V(G)$. Fix $1\leq i\leq s$. Since $u_i$ and $w_i$ are adjacent to all vertices outside $\{u_i,w_i\}$ and are nonadjacent to each other, they cannot share a part with any other vertex. Thus, in $R$ they appear either as a single part $\{u_i,w_i\}$ or as two singletons $\{u_i\},\{w_i\}$.

Define an involution $\Psi$ on $V(\BellMulti{n}{G})$ by \emph{flipping all pairs simultaneously}: given $R$,
\begin{itemize}
\item if $\{u_i,w_i\}$ is a part of $R$, replace it by the singletons
  $\{u_i\}$ and $\{w_i\}$ in $\Psi(R)$;
\item if $\{u_i\}$ and $\{w_i\}$ are singleton parts in $R$, replace them
  by the part $\{u_i,w_i\}$ in $\Psi(R)$;
\item leave all other parts unchanged.
\end{itemize}
The resulting partition $\Psi(R)$ is stable with $n$ parts, and $\Psi$ is an involution.

To see that $\Psi$ is an automorphism, let $R\sim_x R'$. If $x\notin\{u_i,w_i\}$ for all $i$, then the configuration of each pair $\{u_i,w_i\}$ is the same in $R$ and $R'$; flipping all pairs in both partitions does not affect the move at $x$, so $\Psi(R)\sim_x\Psi(R')$ with the same multiplicity.

If $x\in\{u_i,w_i\}$, say $x=u_i$, then in any stable partition the only possible nontrivial move of $u_i$ is to toggle the pair $\{u_i,w_i\}$ between
merged and split: $u_i$ cannot join any other part because it is adjacent to all vertices outside $\{u_i,w_i\}$. Thus, $R'$ is obtained from $R$ by such a toggle, and by definition, the same toggle transforms $\Psi(R)$ into $\Psi(R')$. Again adjacency (and multiplicity) is preserved, so $\Psi$ is an
automorphism.

In $P$, each $\{u_i,w_i\}$ is a $2$-element part. Applying $\Psi$ to $P$ splits all these pairs and leaves the singleton parts $\{v_j\}$ unchanged, so $\Psi(P)=Q$.
\end{proof}

\subsection{The neighborhood of the singleton partition}

Having just determined how to identify (up to automorphism) the singleton partition in the Bell multigraph, we prove that the neighborhood of such a partition is the line graph of the base graph.

\begin{prop}\label{prop:NQ-is-linegraph}
The induced subgraph $\B_n(G)[N(Q)]$ is isomorphic to the line graph
$L(\overline{G})$.
\end{prop}

\begin{proof}
Every neighbor of $Q$ is obtained by merging two singleton parts corresponding to a nonedge of $G$: for $uv\in E(\overline{G})$ let $Q_{uv}$ be the partition in which $\{u,v\}$ is a part and all other vertices are singletons. The map
\[
\Phi\colon N(Q)\to E(\overline{G}),\qquad \Phi(Q_{uv})=uv
\]
is a bijection.

Two neighbors $Q_{uv}$ and $Q_{xy}$ are adjacent in $\B_n(G)$ if and only if they differ by a single move. This occurs if and only if the pairs $\{u, v\}$ and $\{x, y\}$ intersect. If $\{u, v\}\cap \{x, y\}\neq \emptyset$, say $v=y$, then moving $v$ from $\{u,v\}$ to the singleton $\{x\}$ transforms $Q_{uv}$ into $Q_{xv}$, so $Q_{uv}\sim_{v} Q_{xv}$. Conversely, any adjacency $Q_{uv}\sim_{w} Q_{xy}$ forces $w\in \{u, v\}\cap \{x, y\}$. Thus, $Q_{uv}$ and $Q_{xy}$ are adjacent exactly when the vertices $uv$ and $xy$ are adjacent in $L(\overline{G})$, and hence $\B_n(G)[N(Q)]\cong L(\overline{G})$.
\end{proof}

\subsection{Reconstruction from line graphs}

We now complete the reconstruction of $G$. By Lemma~\ref{lem:involution-max-degree}, we can identify the singleton partition $Q$ up to automorphism. By Proposition~\ref{prop:NQ-is-linegraph}, we can construct the line graph $L(\overline{G})$. We then appeal to Whitney's Isomorphism Theorem~\cite{Whi32}, which states that connected graphs with isomorphic line graphs are themselves isomorphic, with the single exception of the triangle $K_3$ and the claw $K_{1,3}$; indeed, $L(K_3) \cong L(K_{1,3}) \cong K_3$. Thus, we can reconstruct $\overline{G}$ from $L(\overline{G})$ uniquely, except when a connected component of $\overline{G}$ is isomorphic to one of these two graphs.

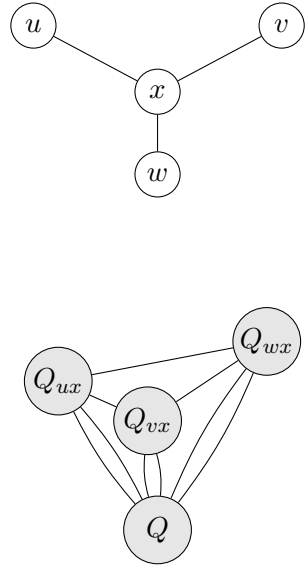
\begin{figure}[ht]
    \centering
    \begin{tikzpicture}[
        gvertex/.style={
            circle,
            draw,
            inner sep=1pt,
            minimum size=6mm, 
            font=\small
        },
        vertex/.style={
            circle,
            draw,
            inner sep=1pt,
            minimum size=9mm, 
            font=\small
        },        g_vertex/.style={gvertex},
        b_vertex/.style={vertex, fill=gray!20},
        scale=1.1,
        xyz_coords/.style={x={(-0.6cm, -0.4cm)}, y={(1cm, 0cm)}, z={(0cm, 1cm)}}
    ]

    \begin{scope}[local bounding box=K3]
        \node at (0, 4) {Case 1: $\overline{G}$ component is $K_3$};

        \node[g_vertex] (u1) at (0, 3.0) {$u$};
        \node[g_vertex] (v1) at (-1, 2.0) {$v$};
        \node[g_vertex] (w1) at (1, 2.0) {$w$};
        \draw (u1)--(v1)--(w1)--(u1);

        \begin{scope}[yshift=-1cm, xyz_coords, scale=0.6]
           
            \coordinate (C_mid1) at (0, -2, 0);
            \coordinate (C_mid2) at (2, 1, 0);
            \coordinate (C_mid3) at (-2, 1, 0);
            \coordinate (C_top) at (0, 0, 3);
            \coordinate (C_bot) at (0, 0, -3);

            \node[b_vertex] (Q1) at (C_bot) {$Q$};

            \node[b_vertex] (P1uv) at (C_mid1) {$Q_{uv}$};
            \node[b_vertex] (P1vw) at (C_mid2) {$Q_{vw}$};
            \node[b_vertex] (P1wu) at (C_mid3) {$Q_{wu}$};

            \draw (Q1) to[bend left=7] (P1uv); \draw (Q1) to[bend right=7] (P1uv);
            \draw (Q1) to[bend left=10] (P1vw); \draw (Q1) to[bend right=10] (P1vw);
            \draw (Q1) to[bend left=7] (P1wu); \draw (Q1) to[bend right=7] (P1wu);

            \draw (P1uv)--(P1vw);
            \draw (P1vw)--(P1wu);
            \draw (P1wu)--(P1uv);

            \node[b_vertex, draw] (R1) at (C_top) {$Q_{uvw}$};
            
            \draw (R1)--(P1uv); \draw (R1)--(P1vw); \draw (R1)--(P1wu);

            \node[font=\footnotesize, color=red, right=0.2cm of R1] {};
        \end{scope}
    \end{scope}

    \begin{scope}[xshift=6.5cm, local bounding box=K13]
        \node at (0, 4) {Case 2: $\overline{G}$ component is $K_{1,3}$};

        \node[g_vertex] (u2) at (0,   2.5) {$x$};
        \node[g_vertex] (v2) at (-1.5,3.3) {$u$};
        \node[g_vertex] (w2) at (0,   1.5) {$w$};
        \node[g_vertex] (x2) at (1.5, 3.3) {$v$};
        \draw (u2)--(v2); \draw (u2)--(w2); \draw (u2)--(x2);

         \begin{scope}[yshift=-1cm, xyz_coords, scale=0.6]
       
            \coordinate (C_mid1) at (0, -2, 0);
            \coordinate (C_mid2) at (2, 1, 0);
            \coordinate (C_mid3) at (-2, 1, 0);
            \coordinate (C_top) at (0, 0, 3);
            \coordinate (C_bot) at (0, 0, -3);
            \node[b_vertex] (Q1) at (C_bot) {$Q$};
            \node[b_vertex] (P1uv) at (C_mid1) {$Q_{ux}$};
            \node[b_vertex] (P1vw) at (C_mid2) {$Q_{vx}$};
            \node[b_vertex] (P1wu) at (C_mid3) {$Q_{wx}$};

            \draw (Q1) to[bend left=7] (P1uv); \draw (Q1) to[bend right=7] (P1uv);
            \draw (Q1) to[bend left=10] (P1vw); \draw (Q1) to[bend right=10] (P1vw);
            \draw (Q1) to[bend left=7] (P1wu); \draw (Q1) to[bend right=7] (P1wu);
=
            \draw (P1uv)--(P1vw);
            \draw (P1vw)--(P1wu);
            \draw (P1wu)--(P1uv); 
            \node[font=\footnotesize, color=red, right=0.2cm of R1] {};
        \end{scope}
    \end{scope}

    \end{tikzpicture}
    \caption{The presence or absence of a common neighbor $Q_{uvw}$ (distinct from $Q$) distinguishes the $K_3$ case from the $K_{1,3}$ case.}
    \label{fig:whitney_ambiguity}
\end{figure}

The following lemma, whose main idea is illustrated in Figure~\ref{fig:whitney_ambiguity} resolves this ambiguity by examining the common neighbors of the corresponding triangles in $\BellMulti{n}{G}$.

\begin{lemma}[Distinguishing $K_3$ and $K_{1,3}$]\label{lem:distinguish-triangles}
Let $T$ be a triangle in $\B_n(G)[N(Q)]$, identified with a triangle in $L(\overline{G})$ via Proposition~\ref{prop:NQ-is-linegraph}.
\begin{enumerate}
    \item If $T$ corresponds to a $K_3$ component in $\overline{G}$, then the vertices of $T$ share a common neighbor in $\BellMulti{n}{G}$ distinct from $Q$.
    \item If $T$ corresponds to a $K_{1,3}$ component in $\overline{G}$, then $Q$ is the unique common neighbor of $T$ in $\BellMulti{n}{G}$.
\end{enumerate}
\end{lemma}

\begin{proof}
(1) Suppose the component is a $K_3$ on vertices $\{u,v,w\}$. Then $\{u,v,w\}$ is an independent set in $G$. Consequently, the partition
\[
  Q_{uvw} = \{\{u,v,w\}\} \cup \{\{z\} : z \notin \{u,v,w\}\}
\]
is stable. Observe that $Q_{uvw}$ is adjacent to $Q_{uv}$ (via $w$), to $Q_{vw}$ (via $u$), and to $Q_{wu}$ (via $v$). Since $Q_{uvw} \neq Q$, the vertices of $T$ share a second common neighbor.

(2) Suppose the component is a $K_{1,3}$ with edges $ux, vx, wx$. Then $T=\{Q_{ux}, Q_{vx}, Q_{wx}\}$. Since $K_{1, 3}$ is induced in $\overline{G}$, the set $\{u, v, w\}$ forms a clique in $G$. Furthermore, as this $K_{1, 3}$ is a component of $\overline{G}$, we have $N_G(u)=V(G)\setminus \{u, v, w, x\}$. Assume, to the contrary, that $T$ has a common neighbor $R$ other than $Q$. Then $\mathcal{K} = \{Q_{ux}, Q_{vx}, Q_{wx}, R\}$ forms a clique of size $4$ in $\BellMulti{n}{G}$.

We first rule out the possibility that $\mathcal{K}$ is a $T$-clique. By Theorem~\ref{thm:clique_classification}, any $T$-clique on four vertices must contain a $T$-triangle. First, observe that $Q_{ux}-x = Q-x = Q_{vx}-x$. Thus, the edge $Q_{ux}Q_{vx}$ is realized by $x$. By symmetry, every edge in $T$ is realized by $x$, so $T$ is an $S$-triangle, not a $T$-triangle.

Now, consider any other triangle in $\mathcal{K}$, say $\{Q_{ux}, Q_{vx}, R\}$. The edge $Q_{ux}Q_{vx}$ connects a partition containing $\{u,x\}$ to one containing $\{v,x\}$. By Lemma~\ref{lem:triple_vertex}, if this were a $T$-triangle, the third partition $R$ would require a non-singleton part $\{u,v\}$. However, $uv$ is an edge in $G$ (as $\{u,v,w\}$ is a clique), so $\{u,v\}$ cannot be a part in a stable partition. Thus, $\mathcal{K}$ contains no $T$-triangles, so $\mathcal{K}$ is not a $T$-clique.

Therefore, $\mathcal{K}$ must be an $S$-clique. By Lemma~\ref{lem:anchor_uniqueness}, there exists a unique anchor $a$ such that all edges in $\mathcal{K}$ are realized by $a$. Since $Q_{ux}$ and $Q_{vx}$ differ only by the placement of $x$, the anchor must be $x$. Since $R$ is anchored at $x$, we have $R-x = Q_{ux}-x = Q-x$. This implies $R$ consists of singletons on $V(G)\setminus \{x\}$. The vertex $x$ must be placed in a part compatible with its neighbors in $G$. Since $x$ is adjacent to all $z \notin \{u,v,w\}$, $x$ can only form a part with subsets of $\{u,v,w\}$. Because $\{u,v,w\}$ is a clique, at most one of these vertices can share a part with $x$, so the only stable partitions of this form are
\[
  Q,\quad Q_{ux},\quad Q_{vx},\quad Q_{wx}.
\]
Thus, $R$ cannot be a new common neighbor, a contradiction.
\end{proof}

We are now ready to prove the main reconstruction result for Bell coloring multigraphs. As mentioned in the introduction, a core of the graph $G$ is the subgraph obtained by removing all universal vertices (i.e., vertices of degree $|V(G)|-1$).

\begin{theorem}\label{thm:general-reconstruction-multigraph:intro}
Suppose $G_1$ and $G_2$ are two graphs of orders $n_1$ and $n_2$, respectively. If $\BellMulti{n_1}{G_1} \cong \BellMulti{n_2}{G_2}$, then $G_1$ and $G_2$ have isomorphic cores.
\end{theorem}

\begin{proof}
We identify the set of totally-doubled partitions $\mathcal{D}(G)$ and locate the singleton partition $Q$ as an element of maximum degree in $\mathcal{D}(G)$ (unique up to automorphism). The neighborhood of $Q$ in $\B_n(G)$ is isomorphic to $L(\overline{G})$ by Proposition~\ref{prop:NQ-is-linegraph}. By Whitney's theorem, $L(\overline{G})$ determines the components of $\overline{G}$ up to the isomorphism $L(K_3) \cong L(K_{1,3})$. Lemma~\ref{lem:distinguish-triangles} resolves this ambiguity by checking whether the associated triangle in $\B_n(G)[N(Q)]$ has a common neighbor distinct from $Q$; if it does, then that component of $\overline{G}$ is $K_3$, and it is $K_{1,3}$ otherwise (see Figure~\ref{fig:whitney_ambiguity}). Since isolated vertices in $\overline{G}$ do not appear in $L(\overline{G})$, we reconstruct $\overline{G}$ up to isolated vertices, or equivalently, $G$ up to universal vertices.
\end{proof}

\bibliographystyle{alpha}
\bibliography{main}

\end{document}